\documentclass[smallcondensed]{svjour3}
\smartqed
\usepackage{amsmath,amssymb,graphicx,amsfonts,bm,algorithm,hyperref,xcolor}
\usepackage[noend]{algorithmic}

\journalname{Journal of Scientific Computing}

\begin{document}

\title{Constrained high-index saddle dynamics for the solution landscape with equality constraints \thanks{This work was supported by the National Natural Science Foundation of China (Grants No.~12050002), National Key R\(\&\)D Program of China 2021YFF1200500, and the Royal Society Newton Advanced Fellowship.}}
\titlerunning{Constrained high-index saddle dynamics}

\author{Jianyuan Yin \and Zhen Huang \and Lei Zhang}
\authorrunning{J. Yin, Z. Huang, L. Zhang}
\institute{J. Yin \at
              School of Mathematical Sciences, Peking University, Beijing 100871, China\\
              \email{yinjy@pku.edu.cn}
           \and
           Z. Huang \at
              School of Mathematical Sciences, Peking University, Beijing 100871, China\\
              \email{hertz@pku.edu.cn}
           \and
           L. Zhang \at
              Beijing International Center for Mathematical Research, Center for Quantitative Biology, Peking University, Beijing 100871, China\\
              \email{zhangl@math.pku.edu.cn}
}

\date{}
\maketitle
\begin{abstract}
  We propose a constrained high-index saddle dynamics (CHiSD) method to search for index-$k$ saddle points of an energy functional subject to equality constraints.
  With Riemannian manifold tools, the CHiSD is derived in a minimax framework, and its linear stability at an index-$k$ saddle point is proved.
  To ensure the manifold property, the CHiSD is numerically implemented using retractions and vector transport.
  Then we present a numerical approach by combining CHiSD with downward and upward search algorithms to construct the solution landscape in the presence of equality constraints.
  We apply the Thomson problem and the Bose--Einstein condensation as numerical examples to demonstrate the efficiency of the proposed method.
\end{abstract}

\section{Introduction}\label{sec:intro}
The energy landscape, which maps all possible configurations of a system to their corresponding energy \cite{wales2003energy}, has been successfully applied to various scientific issues, such as particle clusters \cite{mehta2016kinetic,meng2010free}, protein folding \cite{leeson2000protein,mallamace2016energy}, and soft matter \cite{cheng2010nucleation,han2019transition}.
The stationary point, at which the gradient vanishes on the energy landscape, plays an important role to determine physical or chemical properties of the system.
The stability of a stationary point is determined by its Hessian matrix.
For instance, a stationary point is a minimizer if all eigenvalues of its Hessian are positive.
The stationary points with both positive and negative eigenvalues are called saddle points, which can be further classified by the Morse index.
The Morse index is equal to the number of negative eigenvalues of the Hessian matrix at each stationary point.
In particular, the index-1 saddle point, which has one and only one negative eigenvalue, is referred to as the transition state connecting two minima on the energy landscape.

Searching for saddle points on a complicated energy landscape has attracted plenty of attention during the past decades \cite{zhang2008mathematical,weinan2010transition,zhang2016recent}.
Compared to finding stable minima, computation of saddle points is much more challenging due to their unstable nature.
Extensive numerical algorithms have been developed to compute index-1 saddle points, including path-finding methods \cite{henkelman2000improved,weinan2002string} and surface-walking methods \cite{olsen2004comparison}.
In particular, the surface-walking method starts from one initial state on the energy landscape and searches for index-1 saddle points based on local derivative information without a priori knowledge of the final state.
Examples of the surface-walking methods include the dimer-type methods \cite{henkelman1999dimer,zhang2012shrinking,zhang2016optimization}, the gentlest ascent dynamics \cite{weinan2011gentlest,quapp2014locating}, and the activation-relaxation technique \cite{cances2009some,machado2011optimized}.
Moreover, the surface-walking methods can be generalized to search for high-index (i.e. index greater than one) saddle points.
Quapp and Bofill developed a generalized gentlest ascent dynamics to locate high-index saddle points on the energy landscape with the calculation of the Hessian matrix \cite{quapp2014locating}.
The minimax method based on the local minimax theorem was proposed to find multiple high-index saddle points with a priori knowledge of low-index saddle points \cite{li2001minimax,li2019local1,li2019local2}.
We refer to some excellent reviews for more information \cite{weinan2010transition,zhang2016recent}.
Recently, Yin et al. proposed a high-index saddle dynamics (HiSD) and developed a high-index optimization-based shrinking dimer method for finding index-$k$ saddle points on the energy landscape \cite{yin2019high}.
Later, a generalized high-index saddle dynamics (GHiSD) was developed to compute any-index saddle points of dynamical systems \cite{yin2020searching}.

In many practical applications, the challenge of searching for saddle points is further increased by nonlinear equality constraints on the state variables.
For example, in the Oseen--Frank theory for nematic liquid crystals, the director $\mathbf{n}(\bm x)$ is a vector field subject to a unit-length constraint almost everywhere that describes the average orientation of liquid crystal molecules at the position $\bm x$ \cite{frank1958liquid,wang2021modeling}.
In the Kohn--Sham density functional theory, the electron orbitals $\Psi$ are supposed to satisfy the orthonormality constraint \cite{lin2019numerical}.
A simple approach to deal with nonlinear equality constraints is reparametrization via unconstrained variables, but this is often cumbersome and computationally inefficient.
Therefore, a number of numerical methods have been proposed to compute index-1 saddle points in the presence of equality constraints.
As examples of path-finding methods, the constrained string method \cite{du2009constrained} and the geodesic nudged elastic band method \cite{bessarab2015method} are able to find constrained minimum energy path and index-1 saddle points.
Alternatively, the constrained shrinking dimer dynamics \cite{zhang2012constrained} was developed to locate index-1 saddle points on a constrained energy function using the projected Hessian.
M{\"u}ller et al. identified multiple transitions of Skyrmions in magnetic systems using a surface-walking method, where each magnetic vector is restricted to a unit length \cite{muller2018duplication}.
The physical space, which is the direct product of $N$ spheres, is naturally embedded into the Euclidean space, and the theory of Riemannian manifolds is considered to derive the Hessian.
The unphysical degrees of freedom in the embedding space are removed to calculate the true eigenvectors required in this saddle-point searching method. However, how to numerically compute the constrained high-index (index$>1$) saddle points is still unclear at present.

In this article, we present a constrained high-index saddle dynamics (CHiSD) to search for high-index saddle points subject to equality constraints.
The CHiSD is a constrained version of the HiSD \cite{yin2019high} and derived with Riemannian gradients and Hessians. The retraction operator and vector transport are introduced to implement the numerical algorithm of CHiSD.
With the CHiSD algorithm, we are able to construct the solution landscape on a constrained manifold.
The solution landscape is a pathway map consisting of all stationary points and their connections, which not only provides an efficient approach to find multiple stationary points, including both minima and saddle points, without tuning initial guesses, but also shows the relationships between different stationary points \cite{yin2020construction}.

The rest of this article is organized as follows.
In Sect.~\ref{sec:chisd}, we first introduce the constrained saddle points and Riemannian manifold tools.
After briefly reviewing the HiSD for index-$k$ saddle points, the CHiSD is derived with Riemannian gradient and Hessian in a similar manner, and the linear stability of the index-$k$ saddle point is proved.
In Sect.~\ref{sec:numericalimplementation}, retractions and vector transport are applied to maintain the manifold property in numerical schemes, and then the CHiSD algorithm is presented.
Furthermore, the solution landscape in equality-constrained cases can be constructed with the downward and upward search algorithms based on the CHiSD algorithm.
In Sect.~\ref{sec:numericalexamples}, we show two numerical examples to demonstrate the efficiency of the CHiSD algorithm.
First, we construct the solution landscape of the Thomson problem to identify all possible stationary points in the cases of the particle number $N=5, 7$ and $9$.
Second, as a nonlinear elliptic eigenvalue problem, the excited states of the Bose--Einstein condensates (BEC) are calculated with a combination of the CHiSD and the upward search algorithm.
Some conclusions and discussions are presented in Sect.~\ref{sec:conclusion}.

\section{Constrained high-index saddle dynamics}
\label{sec:chisd}
\subsection{Constrained saddle points}
\label{sec:saddle}
Given a twice Fr\'{e}chet differentiable energy functional $E(\bm{x})$ defined on a $d$-dimensional real Hilbert space $\mathcal{H}$ with an inner product $\left\langle\cdot,\cdot\right\rangle$ and the norm $\|\cdot\|$, we let $\nabla E(\cdot): \mathcal{H}\rightarrow\mathcal{H}$ denotes the Riesz map applied to the Fr\'{e}chet derivative, and $\nabla^2 E(\bm x)\in\mathcal{L}(\mathcal{H})$ denotes the Hessian.
We consider the functional $E(\bm x)$ for $\bm x\in \mathcal{H}$ subject to $m$ equality constraints,
\begin{equation}\label{eqn:constraints}
  \bm c(\bm x)=\left(c_1(\bm x), \cdots, c_m(\bm x)\right)=(0,\cdots,0).
\end{equation}
where each $c_p:\mathcal{H}\rightarrow\mathbb{R}$ is a smooth function.
Let $\mathbf{A}(\bm x)$ denote $\left(\nabla c_1(\bm x), \cdots, \nabla c_m(\bm x)\right)$, and it is always assumed that $\nabla c_1(\bm x)$, $\cdots$, $\nabla c_m(\bm x)$ are linearly independent for each $\bm{x}$ subject to \eqref{eqn:constraints}, which is often referred to as the linear independence constraint qualification (LICQ) in optimization theory \cite{nocedal2006numerical}.
From the regular level set theorem \cite{loring2011introduction}, the feasible set consisting of all feasible points,
\begin{equation}\label{eqn:manifold}
  \mathcal{M}=\{\bm x\in \mathcal{H}: \bm c(\bm x)=(0,\cdots,0)\},
\end{equation}
is a $(d-m)$-dimensional smooth Riemannian manifold with induced metric.
For $\bm x\in \mathcal{H}$, the normal space of the isosurface of $\bm c(\bm x)$ is defined as $N(\bm x)= \operatorname{span}\left\{\mathbf{A}(\bm x)\right\}$, and the tangent space is defined as its orthogonal complement $T(\bm x)={N}(\bm x)^\perp=\{\bm v\in\mathcal{H}:\mathbf{A}(\bm x)^\top\bm v=\bm 0\}$.
For $\mathbf{A}=(\bm a_1, \cdots, \bm a_p)$ and $\mathbf{B}=(\bm b_1, \cdots, \bm b_q)$ with columns in $\mathcal{H}$, $\mathbf{A}^\top \mathbf{B}$ denotes a $p\times q$ matrix whose $(i,j)$-entry is $\langle\bm b_j, \bm a_i\rangle$, and $\mathbf{A}^\top$ denotes the corresponding linear operator for simplicity.
The orthogonal projection operators $\mathbf{P}$ on these spaces have the forms of,
\begin{equation}\label{eqn:projection}
  \mathbf{P}_{N(\bm x)}=\mathbf{A}(\bm x)\left(\mathbf{A}(\bm x)^\top \mathbf{A}(\bm x)\right)^{-1}\mathbf{A}(\bm x)^\top,
  \qquad
  \mathbf{P}_{T(\bm x)}=\mathbf{I}-\mathbf{P}_{N(\bm x)},
\end{equation}
where $\mathbf{A}^\top \mathbf{A}$ is positive definite for each $\bm x\in \mathcal{M}$ from LICQ, and remains positive definite in a neighbourhood $U(\bm x)\subset \mathcal{H}$.

Since the functional $E$ is constrained on the Riemannian manifold $\mathcal{M}$, the Riemannian gradient and Hessian should be considered in replacement of $\nabla E(\bm x)$ and $\nabla^2 E(\bm x)$ in unconstrained cases.
For $\bm x\in\mathcal{M}$, the Riemannian gradient is defined as,
\begin{equation}\label{eqn:riemanniangrad}
  \operatorname{grad}E(\bm x)=\mathbf{P}_{T(\bm x)}\nabla E(\bm x),
\end{equation}
and the Riemannian Hessian $\operatorname{Hess}E(\bm x): \bm \eta\in T(\bm x) \rightarrow T(\bm x)\subset \mathcal{H}$ is defined as,
\begin{equation}\label{eqn:riemannianhess}
\begin{aligned}
  \operatorname{Hess}E&(\bm x)[\bm \eta]=\mathbf{P}_{T(\bm x)} \left(\partial_{\bm \eta} \operatorname{grad}E(\bm x)\right)\\
  &=\mathbf{P}_{T(\bm x)} \left(\nabla^2 E(\bm x) \bm \eta
  -\nabla^2 \bm c(\bm x) \bm \eta
  \left(\mathbf{A}(\bm x)^\top \mathbf{A}(\bm x)\right)^{-1}\mathbf{A}(\bm x)^\top \nabla E(\bm x)\right),
\end{aligned}
\end{equation}
where $\nabla^2 \bm c(\bm x) \bm \eta=\left(\nabla^2 c_1(\bm x) \bm \eta, \cdots, \nabla^2 c_m(\bm x) \bm \eta\right)$ \cite{absil2008optimization}.
The Riemannian Hessian \eqref{eqn:riemannianhess} is a symmetric operator on $T(\bm x)$, and can be extended as a symmetric operator on $\mathcal{H}$,
\begin{equation}\label{eqn:riemannianhessextend}
  \operatorname{\widetilde Hess}E(\bm x)[\bm \eta]:=\operatorname{Hess}E(\bm x)[\mathbf{P}_{T(\bm x)}\bm \eta],
\end{equation}
to avoid the problem of definition domains.
The Riemannian gradient \eqref{eqn:riemanniangrad} and Hessian \eqref{eqn:riemannianhess} can be naturally extended to a neighbourhood $U(\bm x)\subset \mathcal{H}$ of each $\bm x \in \mathcal{M}$ with the same expressions, and can be thought of as the Riemannian gradient and Hessian on some isosurface of $\bm c(\bm x)$.

A point $\hat{\bm x}\in\mathcal{M}$ is called a \emph{stationary point} (or \emph{critical point}) of the functional $E$ subject to equality constraints \eqref{eqn:constraints} if $\operatorname{grad}E(\hat{\bm x})=\bm 0$, and a stationary point $\hat{\bm x}$ is said to be \emph{nondegenerate} if $\operatorname{Hess}E(\hat{\bm x})$ has a bounded inverse on $T(\hat{\bm x})$.
A stationary point $\hat{\bm x}\in\mathcal{M}$ is a \emph{(constrained) saddle point}, if $\hat{\bm x}$ is not a local extremum on $\mathcal{M}$.
According to the Morse theory, the \emph{(Morse) index} of a stationary point $\hat{\bm x}$ is defined as the maximal dimension of a subspace $\mathcal{K}\subseteq T(\hat{\bm x})$ on which $\operatorname{Hess}E(\hat{\bm x})$ is negative definite \cite{milnor1963morse}.
For an index-$k$ saddle point (a $k$-saddle) $\hat{\bm x}$, the Riemannian Hessian $\operatorname{Hess}E(\hat{\bm x})$ has $(d-m)$ eigenvalues $\hat{\lambda}_1 \leqslant \cdots \hat{\lambda}_k <0\leqslant \hat{\lambda}_{k+1}\leqslant \cdots \leqslant \hat{\lambda}_{d-m}$ with corresponding orthonormal eigenvectors $\hat{\bm v}_1,\cdots,\hat{\bm v}_{d-m}$.
A nondegenerate index-$k$ saddle point $\hat{\bm x}$ is a local maximum on a $k$-dimensional submanifold $\mathcal{M}_-(\hat{\bm x})$ of $\mathcal{M}$, and a local minimum on a $(d-m-k)$-dimensional submanifold $\mathcal{M}_+(\hat{\bm x})$, and the tangent spaces at $\hat{\bm x}$ of the two submanifolds $\mathcal{M}_-(\hat{\bm x})$ and $\mathcal{M}_+(\hat{\bm x})$ are respectively $\operatorname{span}\{\hat{\bm v}_1,\cdots,\hat{\bm v}_k\}$ and $\operatorname{span}\{\hat{\bm v}_{k+1}, \cdots,\hat{\bm v}_{d-m}\}$.
This minimax structure inspires us to develop numerical methods for searching for saddle points with a certain index on the manifold.

Another frequently-used approach to studying stationary points in a constrained problem is the Lagrangian function.
The Lagrangian function $L_E(\bm x, \bm \xi)$ of the energy $E$ with the equality constraints \eqref{eqn:constraints} is,
\begin{equation}\label{eqn:lagrangianfunction}
L_E(\bm x, \bm \xi)=E(\bm x)-\bm c(\bm x)\bm \xi,
\end{equation}
where $\bm\xi\in \mathbb{R}^m$ is the Lagrangian multiplier.
A stationary point $\hat{\bm x}\in\mathcal{M}$ can be equivalently defined as where the first-order Karush--Kuhn--Tucker (KKT) condition,
\begin{equation}\label{eqn:kkt}
\nabla_{\bm x} L_{E}(\hat{\bm x}, \hat{\bm \xi})=\nabla E(\hat{\bm x})-\mathbf{A}(\hat{\bm x})\hat{\bm \xi}=\bm 0,
\end{equation}
holds for some multiplier $\hat{\bm \xi}\in \mathbb{R}^m$ \cite{nocedal2006numerical}.
From the KKT condition \eqref{eqn:kkt} and LICQ, the multiplier can be calculated as $\hat{\bm\xi}=\left(\mathbf{A}(\hat{\bm x})^\top \mathbf{A}(\hat{\bm x})\right)^{-1} \mathbf{A}(\hat{\bm x})^\top \nabla E(\hat{\bm x})$, and consequently, the Lagrangian Hessian at a stationary point $\hat{\bm x}$ is
\begin{equation}\label{eqn:lagrangianhess}
\nabla_{\bm x \bm x}^2 L_{E}(\hat{\bm x}, \hat{\bm \xi})
=\nabla^2 E(\hat{\bm x})-\nabla^2 \bm c(\hat{\bm x})\left(\mathbf{A}(\hat{\bm x})^\top \mathbf{A}(\hat{\bm x})\right)^{-1}
  \mathbf{A}(\hat{\bm x})^\top \nabla E(\hat{\bm x}).
\end{equation}
Since only the vectors in the tangent space $T(\hat{\bm x})$ are physical directions, a projected Hessian $\mathbf{P}_{T(\hat{\bm x})} \nabla_{\bm x\bm x}^2 L_{E}(\hat{\bm x}, \hat{\bm \xi}) \mathbf{P}_{T(\hat{\bm x})}$ is considered to determine the second-order properties of the stationary point, which accords with \eqref{eqn:riemannianhess} in the tangent space.

We take the unit sphere constraint $c(\bm x)=(\bm x^{\top}\bm x-1)/2$ in the Euclidean space $\mathbb{R}^d$ as a simple example, where the feasible set is customarily denoted as $S^{d-1}=\left\{\bm{x}\in\mathbb{R}^d: \|\bm x\|_2^2=1\right\}$.
For $\bm x\neq\bm 0$ and $\bm \eta \in T(\bm x)=\operatorname{span}\{\bm x\}^\perp$, we have,
\begin{equation}\label{eqn:sphere}
\begin{aligned}
&\mathbf{P}_{N(\bm x)}=\bm x(\bm x^{\top}\bm x)^{-1}\bm x^{\top},\quad \operatorname{grad}E(\bm x)=\left(\mathbf{I}-\bm x(\bm x^{\top}\bm x)^{-1}\bm x^{\top}\right)\nabla E(\bm x),\\
&\operatorname{Hess}E(\bm x)[\bm \eta]=\left(\mathbf{I}-\bm x(\bm x^{\top}\bm x)^{-1}\bm x^{\top}\right)\nabla^2 E(\bm x) \bm \eta
  -(\bm x^{\top}\bm x)^{-1}\bm x^\top \nabla E(\bm x)\bm \eta.
\end{aligned}
\end{equation}

\subsection{Review of the HiSD method}
\label{sec:hisd}
The HiSD method aims to search for a high-index saddle point in an unconstrained case \cite{yin2019high}.
The HiSD for a $k$-saddle is given by
\begin{equation}\label{eqn:hisd}
\left\{
\begin{aligned}
\dot{\bm{x}}  &=-\left(\mathbf{I}-\sum\limits_{i=1}^{k} 2\bm{v}_{i} \bm{v}_{i}^{\top}\right) \nabla {E}(\bm{x}),\\
\dot{\bm{v}_i}&=-\left(\mathbf{I}-\bm{v}_i\bm{v}_i^\top- \sum\limits_{j=1}^{i-1} 2\bm{v}_j\bm{v}_j^\top\right) \nabla^2 E(\bm{x})\bm{v}_i, \quad i=1, \cdots, k,
\end{aligned}
\right.
\end{equation}
which involves a position variable $\bm x$ and $k$ directional variables $\bm v_i$ with an initial condition,
\begin{equation}\label{eqn:hisdinitial}
\bm x=\bm x^{(0)}\in \mathcal{H}, \quad \bm v_i = \bm v_i^{(0)}\in \mathcal{H}, \quad
\mathrm{s.t.}\left\langle\bm v_j^{(0)},\bm v_i^{(0)}\right\rangle=\delta_{ij}, \quad i,j=1,\cdots,k.
\end{equation}
With straightforward calculations, for the HiSD \eqref{eqn:hisd} with the initial condition \eqref{eqn:hisdinitial}, the vectors $\bm v_1,\cdots,\bm v_k$ always satisfy the orthonormal condition $\left\langle\bm v_j,\bm v_i\right\rangle=\delta_{ij}$.
The dynamics for $\bm x$ is actually a transformed gradient flow,
\begin{equation}\label{eqn:hisdx}
\dot{\bm{x}} = \mathbf{P}_{\mathcal{V}}\nabla E(\bm x)-(\mathbf{I}-\mathbf{P}_{\mathcal{V}})\nabla E(\bm x),
\end{equation}
where $\mathbf{P}_{\mathcal{V}}\nabla E(\bm x)$ is the gradient ascent direction on $\mathcal{V}=\operatorname{span}\{\bm v_1, \cdots,\bm v_k\}$ and $(\mathbf{I}-\mathbf{P}_{\mathcal{V}})\nabla E(\bm x)$ is the gradient descent direction on $\mathcal{V}^\perp$.
Since a nondegenerate $k$-saddle is a local maximum along $k$ orthogonal directions and a local minimum along other orthogonal directions, this dynamics can find a $k$-saddle with proper $\{\bm v_i: i=1,\cdots,k\}$.
The dynamics for $\bm v_i$ renews the subspace $\mathcal{V}$ in \eqref{eqn:hisdx} by finding the eigenvectors corresponding to the smallest $k$ eigenvalues of the Hessian $\nabla^2 E(\bm x)$ at the current position $\bm x$.
The eigenvector $\bm v_i$ corresponding to the $i$-th smallest eigenvalue of the Hessian $\nabla^2 E(\bm x)$ can be obtained by solving a constrained optimization problem of the Rayleigh quotient,
\begin{equation}\label{eqn:hisdrayleigh}
\min_{\bm v_i \in\mathcal{H}} \quad \left\langle\nabla^2 E(\bm x) \bm v_i, \bm v_i\right\rangle \quad \mathrm{s.t.}\left\langle\bm v_j,\bm v_i\right\rangle=\delta_{ij},\quad j=1,\cdots,i,
\end{equation}
with the knowledge of $\bm v_1, \cdots,\bm v_{i-1}$,
and the $\bm v_i$ dynamics in \eqref{eqn:hisd} solves the constrained optimization problem \eqref{eqn:hisdrayleigh} using gradient flow.
Then the subspace $\mathcal{V}$ in \eqref{eqn:hisdx} is spanned by the vectors $\{\bm v_i: i=1,\cdots,k\}$.
In practice, the Hessian in \eqref{eqn:hisd} is often approximated by dimers with a length of $2l$ \cite{henkelman1999dimer},
\begin{equation}\label{eqn:dimer}
\nabla^2 E(\bm{x})\bm{v}_i \approx \left(\nabla E(\bm{x}+l\bm{v}_i)-\nabla E(\bm{x}-l\bm{v}_i)\right)/2l,
\end{equation}
and the HiSD with dimer approximations and dimer shrinkage $\dot{l}=-l$ is referred to as the high-index optimization-based shrinking dimer method \cite{yin2019high}.
In numerical implementation, the HiSD can be simply discretized with an explicit Euler scheme.

\subsection{Formulation of CHiSD}
\label{sec:chisddyn}
In order to search for a nondegenerate $k$-saddle $\hat{\bm x}$ on the manifold $\mathcal{M}$, the HiSD method \eqref{eqn:hisd} is supposed to be generalized to the equality-constrained case \eqref{eqn:constraints}.
With the minimax structure of saddle points on manifolds, we set the dynamics for a $k$-saddle $\hat{\bm x}$ as a transformed gradient flow similarly,
\begin{equation}\label{eqn:chisdx}
\dot{\bm x}=-\left(\mathbf{I}-2\mathbf{P}_{\mathcal{V}}\right) \operatorname{grad} E(\bm{x}),
\end{equation}
which is gradient ascent on the subspace $\mathcal{V}$ and gradient descent on its orthogonal complement $\mathcal{V}^\perp$.
The $k$-dimensional subspace $\mathcal{V}\subseteq T(\bm x)$ is spanned by $k$ orthonormal directions $\bm{v}_1$, $\cdots$, $\bm v_k$, and the orthogonal projection operator $\mathbf{P}_{\mathcal{V}}$ possesses a simple form of $\sum_{i=1}^{k} \bm{v}_{i} \bm{v}_{i}^{\top}$.
At a $k$-saddle $\hat{\bm x}$, the $k$ directions $\bm{v}_1$, $\cdots$, $\bm v_k\in T(\hat{\bm x})$ should be the eigenvectors of the smallest $k$ eigenvalues of the Riemannian Hessian $\operatorname{Hess}E(\hat{\bm x})$ \eqref{eqn:riemannianhess}.
Therefore, the direction $\bm v_i$ at the current position $\bm x$ is approximated by the $i$-th eigenvector of $\operatorname{Hess}E(\bm x)$, which can be obtained by a constrained optimization problem with the knowledge of $\bm v_1, \cdots, \bm v_{i-1}$,
\begin{equation}\label{eqn:chisdoptvi}
\min\limits_{\bm v_i\in T(\bm x)} \Big\langle\operatorname{Hess}E(\bm x)[\bm{v}_{i}], \bm{v}_{i}\Big\rangle,
\quad\mathrm{s.t.}\quad\left\langle\bm{v}_j, \bm{v}_i\right\rangle=\delta_{ij},\quad j=1, \cdots, i.
\end{equation}
Equivalently, we deal with another constrained optimization problem,
\begin{equation}\label{eqn:chisdoptviwidetilde}
\begin{aligned}
&\min\limits_{\bm v_i\in \mathcal{H}} \left\langle\operatorname{\widetilde Hess}E(\bm x)[\bm{v}_{i}], \bm{v}_{i}\right\rangle,\\
&\mathrm{s.t.}\quad\mathbf{A}(\bm x)^\top\bm v_i=\bm 0, \quad \left\langle\bm{v}_j, \bm{v}_i\right\rangle=\delta_{ij},\quad j=1, \cdots, i,
\end{aligned}
\end{equation}
with the operator $\operatorname{\widetilde Hess}E(\bm x): \mathcal{H}\rightarrow T(\bm x)\subset \mathcal{H}$ in \eqref{eqn:riemannianhessextend}.
The Lagrangian function of \eqref{eqn:chisdoptviwidetilde} is,
\begin{equation}\label{eqn:Lagarangianvi}
L_i(\bm v_i,\bm \xi_i,\bm \mu_i)=\left\langle\operatorname{\widetilde Hess}E(\bm x)[\bm{v}_{i}], \bm{v}_{i}\right\rangle
-\sum_{j=1}^{i} \xi_{ij}\left(\left\langle\bm{v}_{j}, \bm{v}_{i}\right\rangle-\delta_{ij}\right)
-\bm v_i^\top\mathbf{A}(\bm x)\bm \mu_i,
\end{equation}
where $\bm \xi_i\in \mathbb{R}^i$ and $\bm \mu_i\in\mathbb{R}^m$ are Lagrangian multipliers, and the dynamics of $\bm v_i$ is,
\begin{equation}\label{eqn:chisdvi}
\dot{\bm v}_{i}=-\dfrac12\nabla_{\bm{v}_{i}}L_i(\bm v_i,\bm \xi,\bm \mu)
=-\operatorname{\widetilde Hess}E(\bm x)[\bm{v}_{i}]
+\xi_{i}\bm{v}_{i}+\sum_{j=1}^{i-1} \dfrac{\xi_{ij}}{2} \bm{v}_{j}
+\dfrac12\mathbf{A}(\bm x)\bm \mu_i.
\end{equation}
The dynamics \eqref{eqn:chisdx} and \eqref{eqn:chisdvi} should maintain the manifold property,
\begin{equation}\label{eqn:manifoldproperty}
  \bm x \in \mathcal{M}, \quad \bm{v}_1, \cdots, \bm v_k \in T(\bm x),
\end{equation}
as well as the orthonormal condition,
\begin{equation}\label{eqn:orthonormal}
  \quad\left\langle\bm{v}_j, \bm{v}_i\right\rangle=\delta_{ij},\quad i,j=1, \cdots, k.
\end{equation}
Therefore, the Lagrangian multipliers $\bm \xi_i,\bm \mu_i$ should satisfy,
\begin{equation}\label{eqn:dtconstraints}
\begin{aligned}
  &\dfrac{\mathrm{d}}{\mathrm{d}t}(\bm A(\bm x)^\top\bm v_i)
  =\bm A(\bm x)^\top\dot{\bm v}_i +
   \left(\nabla^2 \bm c(\bm x)\dot{\bm x}\right)^\top \bm v_i= \bm0, \\
  &\dfrac{\mathrm{d}}{\mathrm{d}t}\left\langle \bm v_j, \bm v_i\right\rangle =
  \left\langle \dot{\bm v}_j, \bm v_i\right\rangle+\left\langle \bm v_j, \dot{\bm v}_i\right\rangle=0,
\end{aligned}
\end{equation}
and are obtained as,
\begin{equation}\label{eqn:multipliers}
\begin{aligned}
  &\xi_{ii}=\left\langle\operatorname{\widetilde Hess}E(\bm x)[\bm{v}_{i}], \bm{v}_{i}\right\rangle, \quad
  \xi_{ij}=4\left\langle\operatorname{\widetilde Hess}E(\bm x)[\bm{v}_{i}], \bm{v}_{j}\right\rangle, \\
  &\bm \mu_i=-2\left(\mathbf{A}(\bm x)^\top \mathbf{A}(\bm x)\right)^{-1}
  \left(\nabla^2 \bm c(\bm x) \dot{\bm x}\right)^\top \bm v_{i}.
\end{aligned}
\end{equation}
As a result, we obtain the CHiSD for a $k$-saddle ($k$-CHiSD) as,
\begin{equation}\label{eqn:chisd}
\left\{
\begin{aligned}
\dot{\bm{x}}  =&-\left(\mathbf{I}-\sum\limits_{i=1}^{k} 2\bm{v}_{i} \bm{v}_{i}^{\top}\right) \operatorname{grad} E(\bm{x}),\\
\dot{\bm{v}}_i=&-\left(\mathbf{I}-\bm{v}_i\bm{v}_i^\top- \sum\limits_{j=1}^{i-1} 2\bm{v}_j\bm{v}_j^\top\right) \operatorname{\widetilde Hess}E(\bm x)[\bm{v}_{i}] \\
&-\mathbf{A}(\bm x)\left(\mathbf{A}(\bm x)^\top \mathbf{A}(\bm x)\right)^{-1}
  \left(\nabla^2 \bm c(\bm x) \dot{\bm x}\right)^\top \bm v_{i}, \quad i=1, \cdots, k,
\end{aligned}
\right.
\end{equation}
with an initial condition at $t=0$,
\begin{equation}\label{eqn:chisdinitial}
  \bm x=\bm x^{(0)}\in \mathcal{M}, \quad
  \bm v_i=\bm v_i^{(0)} \in T(\bm x^{(0)}), \quad
  \left\langle\bm v_j^{(0)},\bm v_i^{(0)}\right\rangle=\delta_{ij}, \quad i,j=1,\cdots,k.
\end{equation}

\begin{remark}\label{rmk:consercation}
With straightforward calculations, it can be verified that the CHiSD \eqref{eqn:chisd} with the initial condition \eqref{eqn:chisdinitial} satisfies \eqref{eqn:manifoldproperty} and \eqref{eqn:orthonormal} for $t>0$.
\end{remark}
\begin{remark}\label{rmk:parallel}
The second term in the dynamics of $\bm v_i$ in \eqref{eqn:chisd}, which is derived from the constraints $\bm v_i\in T(\bm x)$, corresponds to the parallel translation induced by the Riemannian connection.
For a smooth curve $\bm x(t)$ on the manifold $\mathcal{M}$, the tangent space $T(\bm x)$ changes accordingly, and the dynamics
\begin{equation}\label{eqn:paralleltranport}
  \dot{\bm{v}}=-\mathbf{A}(\bm x)\left(\mathbf{A}(\bm x)^\top \mathbf{A}(\bm x)\right)^{-1}
  \left(\nabla^2 \bm c(\bm x) \dot{\bm x}\right)^\top \bm v,
\end{equation}
transports the tangent vector $\bm v(0)\in T(\bm x(0))$ to $\bm v(t)\in T(\bm x(t))$ along the curve $\bm x(t)$ on the manifold $\mathcal{M}$ in a parallel way, i.e.
$\frac{\mathrm{D}}{\mathrm{d}t}\bm v(t):=\mathbf{P}_{T(\bm x(t))}\frac{\mathrm{d}}{\mathrm{d}t}\bm v(t)\equiv\bm 0$.
\end{remark}

For the unit sphere $S^{d-1}$, the CHiSD can be simplified as,
\begin{equation}\label{eqn:chisdspheresimplified}
\left\{
\begin{aligned}
\dot{\bm{x}}  =&-\left(\mathbf{I}-\bm x \bm x^\top-\sum\limits_{i=1}^{k} 2\bm{v}_{i} \bm{v}_{i}^{\top}\right) \nabla E(\bm{x}),\\
\dot{\bm{v}}_i=&-\left(\mathbf{I}-\bm x \bm x^\top-\bm{v}_i\bm{v}_i^\top- \sum\limits_{j=1}^{i-1} 2\bm{v}_j\bm{v}_j^\top\right)
\nabla^2 E(\bm x)\bm{v}_{i} -\bm x \bm v_i^\top \nabla E(\bm x), \quad i=1, \cdots, k.
\end{aligned}
\right.
\end{equation}
The sphere constraint is a particularly special case because the Hessian of the constraint $\nabla^2 c(\bm x)$ is a scalar multiple of the identity.
Therefore, $(\mathbf{I}-\bm x \bm x^\top)\nabla^2 E(\bm x)$ and the Riemannian Hessian $\operatorname{Hess}(\bm x)$ share the same eigenvectors in $T(\bm x)$ with a translation in eigenvalues, and the Riemannian Hessian in \eqref{eqn:chisdoptvi} can be replaced by $\nabla^2 E(\bm x)$, which is not valid for general constraints.

\subsection{Linear stability}
Although CHiSD \eqref{eqn:chisd} with the initial condition \eqref{eqn:chisdinitial} satisfies the property \eqref{eqn:manifoldproperty} and \eqref{eqn:orthonormal}, a small perturbation may easily deviate $\bm x$ away from the manifold $\mathcal{M}$.
To achieve better stability, a modified term with $\mu>0$ is attached to the dynamics of $\bm x$ to reinforce the equality constraints \eqref{eqn:constraints}, leading to a modified CHiSD for a $k$-saddle,
\begin{equation}\label{eqn:chisdm}
\left\{
\begin{aligned}
\dot{\bm{x}}  =&-\left(\mathbf{I}-\sum\limits_{i=1}^{k} 2\bm{v}_{i} \bm{v}_{i}^{\top}\right) \operatorname{grad} E(\bm{x})
-\mu \sum_{l=1}^m c_l(\bm x)\nabla c_l(\bm x),\\
\dot{\bm{v}}_i=&-\left(\mathbf{I}-\bm{v}_i\bm{v}_i^\top- \sum\limits_{j=1}^{i-1} 2\bm{v}_j\bm{v}_j^\top\right) \operatorname{\widetilde Hess}E(\bm x)[\bm{v}_{i}] \\
&-\mathbf{A}(\bm x)\left(\mathbf{A}(\bm x)^\top \mathbf{A}(\bm x)\right)^{-1}
  \left(\nabla^2 \bm c(\bm x) \dot{\bm x}\right)^\top \bm v_{i}, \quad i=1, \cdots, k.
\end{aligned}
\right.
\end{equation}
The following theorem shows the linear stability of $k$-saddles in \eqref{eqn:chisdm}.

\begin{theorem}\label{thm}
Assume that $E(\bm x)$ is a $\mathcal{C}^3$ functional, $\bm x^\ast \in \mathcal{M}$, $\left\{\bm v_i^\ast\right\}_{i=1}^k \subset T(\bm x^*)$ satisfies $\|\bm v_i^\ast\|=1$, $\mu>0$, and $\operatorname{Hess}E(\bm x^\ast)$ is nondegenerate, whose eigenvalues are $\lambda^\ast_1<\cdots<\lambda^\ast_k\leqslant\lambda^\ast_{k+1}\leqslant\cdots\leqslant\lambda^\ast_{d-m}$. Then $(\bm x^\ast,\bm v_1^\ast,\cdots,\bm v_k^\ast)$ is a linearly stable steady state of \eqref{eqn:chisdm}, if and only if $\bm x^{\ast}$ is a $k$-saddle and $\operatorname{Hess}E(\bm x^\ast)[\bm v^\ast_i]=\lambda^\ast_i\bm v^\ast_i$ for $i=1,\cdots,k$.
\end{theorem}

\begin{proof}
We consider the Jacobian operator of the dynamics \eqref{eqn:chisdm},
\begin{equation}\label{eqn:chisdmjacobian}
  \mathbb{J}
  =\dfrac{\partial(\dot {\bm x},\dot {\bm v_1},\dot {\bm v_2}, \cdots,\dot {\bm v_k})}
         {\partial( {\bm x}, {\bm v_1}, {\bm v_2}, \cdots, {\bm v_k})}
  =\begin{pmatrix}
      \mathbb{J}_{\bm x} & \mathbb{J}_{\bm x 1} & \mathbb{J}_{\bm x 2} & \cdots & \mathbb{J}_{\bm x k} \\
      \star & \mathbb{J}_{1} & \mathbb{O} & \cdots & \mathbb{O} \\
      \star & \star & \mathbb{J}_{2} & \cdots & \mathbb{O} \\
      \vdots & \vdots & \vdots &  & \vdots \\
      \star & \star & \star & \cdots & \mathbb{J}_{k}\\
    \end{pmatrix},
\end{equation}
whose blocks have the following expressions,
\begin{align*}
    \mathbb{J}_{\bm x} =\dfrac{\partial\dot {\bm x}}{\partial\bm x}
    &=-\left(\mathbf{I}-\sum\limits_{i=1}^{k} 2\bm{v}_{i} \bm{v}_{i}^{\top}\right)
    \left(\mathbf{H}(\bm x)-\mathbf{A}(\bm x)\left(\mathbf{A}(\bm x)^\top \mathbf{A}(\bm x)\right)^{-1}
    \right. \\ &\left. \quad
    \left(\nabla^2 \bm c(\bm x)\operatorname{grad} E(\bm{x})\right)^\top \right)
    -\mu\sum_{l=1}^m \left(c_l(\bm x)\nabla^2 c_l(\bm x)+ \nabla c_l(\bm x) \nabla c_l(\bm x)^\top\right),
\\
    \mathbb{J}_{\bm x i}=\dfrac{\partial\dot {\bm x}}{\partial\bm v_i}
    &=2\bm v_i^\top \operatorname{grad }E(\bm{x}) \mathbf{I}+2\bm v_i \operatorname{grad }E(\bm{x})^\top,
\\
    \mathbb{J}_{i} =\dfrac{\partial\dot {\bm v}_i}{\partial\bm v_i}
    &=-\left(\mathbf{I}-\sum\limits_{j=1}^{i} 2\bm{v}_j\bm{v}_j^\top\right) \operatorname{\widetilde Hess}E(\bm x)
    +\left\langle \bm{v}_i, \operatorname{\widetilde Hess}E(\bm x)[\bm{v}_i]\right\rangle \mathbf{I}\\
    &\quad -\mathbf{A}(\bm x)\left(\mathbf{A}(\bm x)^\top \mathbf{A}(\bm x)\right)^{-1}
    \left(\nabla^2 \bm c(\bm x) \dot{\bm x}
    + \mathbb{J}_{\bm x i}^\top \nabla^2 \bm c(\bm x) \bm{v}_i \right)^\top.
\end{align*}
Here, $\mathbf{H}(\bm x)$ is an asymmetric extension of the Riemannian Hessian \eqref{eqn:riemannianhess} defined as
\begin{equation}\label{eqn:chisdhmatrix}
\mathbf{H}(\bm x)=\mathbf{P}_{T(\bm x)} \left(\nabla^2 E(\bm x)
  -\nabla^2 \bm c(\bm x)
  \left(\mathbf{A}(\bm x)^\top \mathbf{A}(\bm x)\right)^{-1}\mathbf{A}(\bm x)^\top \nabla E(\bm x)\right).
\end{equation}
In the following, $\mathbb{J}(\bm x^\ast,\bm v_1^\ast,\cdots,\bm v_k^\ast)$ is denoted as $\mathbb{J}^\ast$ and the blocks of $\mathbb{J}^\ast$ are denoted as $\mathbb{J}^\ast$ with corresponding subscripts.

"$\Leftarrow$":
Supposing that $\bm x^{\ast}$ is a $k$-saddle and $\operatorname{Hess}E(\bm x^\ast)[\bm v_i^\ast]=\lambda_i^\ast \bm v_i^\ast$ for $i=1,\cdots,k$, we have $\operatorname{grad}E(\bm x^*)=\bm 0$ and $\lambda_k^\ast <0 <\lambda_{k+1}^\ast$.
With simple calculations, $(\bm x^\ast,\bm v_1^\ast,\cdots,\bm v_k^\ast)$ is a steady state of \eqref{eqn:chisdm}.
Note that $\mathbb{J}_{\bm x i}$ is null if $\operatorname{grad}E(\bm x)=\bm 0$, so the Jacobian \eqref{eqn:chisdmjacobian} is block lower triangular, whose eigenvalues are determined by the diagonal blocks,
\begin{align*}
    \mathbb{J}_{\bm x}^\ast
    &=-\left(\mathbf{I}-\sum\limits_{i=1}^{k} 2\bm{v}_{i}^\ast  {\bm{v}_{i}^\ast}^{\top}\right)
    \mathbf{H}(\bm x^\ast)-\mu \sum_{l=1}^m \nabla c_l(\bm x^\ast) \nabla c_l(\bm x^\ast )^\top,
\\
    \mathbb{J}_{i}^\ast
    &=-\left(\mathbf{I}-\sum\limits_{j=1}^{i} 2\bm{v}_j^\ast {\bm{v}_j^\ast} ^\top\right) \operatorname{\widetilde Hess}E(\bm x^\ast )
    +\lambda_i^\ast\mathbf{I}.
\end{align*}

For $\operatorname{Hess}E(\bm x^\ast)$, we denote $\bm v_i^\ast$ $(i=k+1,\cdots,d-m)$ as the eigenvector of its eigenvalue $\lambda_i^\ast$ so that $\{\bm v_i^\ast\}_{i=1}^{d-m}$ is an orthonormal basis of $T(\bm x^\ast)$.
The equations
\begin{equation}\label{eqn:chisdjx1}
  \mathbb{J}_{\bm x}^\ast\bm{v}_{j}^\ast= -\left(\mathbf{I}-\sum\limits_{i=1}^{k} 2\bm{v}_{i}^\ast  {\bm{v}_{i}^\ast}^{\top}\right)\operatorname{Hess}E(\bm x^\ast)[\bm v_j^\ast]=
\left\{
\begin{aligned}
 \lambda_j^\ast\bm v_j^\ast,  &\quad 1\leqslant j\leqslant k,\\
 -\lambda_j^\ast\bm v_j^\ast, &\quad k<j\leqslant d-m,
\end{aligned}
\right.
\end{equation}
indicate that $\mathbb{J}_{\bm x}^\ast$ has eigenvalues of $\lambda_1^\ast, \cdots, \lambda_k^\ast$, $-\lambda_{k+1}^\ast, \cdots, -\lambda_{d-m}^\ast$, and the equations
\begin{equation}\label{eqn:chisdjx2}
  {\nabla c_j(\bm x^\ast)}^\top\mathbb{J}_{\bm x}^\ast=-\mu \|{\nabla c_j(\bm x^\ast)}\|^2{\nabla c_j(\bm x^\ast)}^\top,\quad 1\leqslant j\leqslant m.
\end{equation}
indicate that $\mathbb{J}_{\bm x}^\ast$ has $m$ eigenvalues of $-\mu \|{\nabla c_1(\bm x^\ast)}\|^2$, $\cdots$, $-\mu \|{\nabla c_m(\bm x^\ast)}\|^2$, which are different from the previous ones due to $\mu$.
Therefore, all the eigenvalues of $\mathbb{J}_{\bm x}^\ast$ are negative.
For $\mathbb{J}_{i}^\ast$, we have
\begin{equation}\label{eqn:chisdji1}
  \mathbb{J}_{i}^\ast\bm{v}_{j}^\ast=
\left\{
\begin{aligned}
 (\lambda_i^\ast+\lambda_j^\ast)\bm v_j^\ast,  &\quad 1\leqslant j\leqslant i,\\
 (\lambda_i^\ast-\lambda_j^\ast)\bm v_j^\ast,  &\quad i<j\leqslant d-m,
\end{aligned}
\right.
\end{equation}
indicating eigenvalues of $\lambda_i^\ast+\lambda_1^\ast$, $\cdots$, $\lambda_i^\ast+\lambda_i^\ast$, $\lambda_i^\ast-\lambda_{i+1}^\ast$, $\cdots$, $\lambda_i^\ast-\lambda_{d-m}^\ast$, and
\begin{equation}\label{eqn:chisdji2}
\mathbb{J}_{i}^\ast \nabla c_j(\bm x^\ast)=\lambda_i^\ast {\nabla c_j(\bm x^\ast)},\quad 1\leqslant j\leqslant m,
\end{equation}
indicating an eigenvalues $\lambda_i^\ast$ with multiplicity $m$.
Therefore, all the eigenvalues of $\mathbb{J}_{i}^\ast$ are negative, and the steady state $(\bm x^\ast,\bm v_1^\ast,\cdots,\bm v_k^\ast)$ is linearly stable.

"$\Rightarrow$":
Supposing that $(\bm x^\ast,\bm v_1^\ast,\cdots,\bm v_k^\ast)$ is a linearly stable steady state, we have $\dot{\bm x}=\bm 0$ and $\dot{\bm v}_i=\bm 0$, indicating
\begin{equation}\label{eqn:vidotequ0}
\left(\mathbf{I} - \sum\limits_{j=1}^{i-1} 2\bm{v}_j^\ast{\bm{v}_j^\ast}^\top\right) \operatorname{\widetilde Hess}E(\bm x^\ast)[\bm{v}_{i}^\ast]= \mu_i^\ast \bm{v}_i^\ast,
\end{equation}
where $\mu_i^\ast=\left\langle \operatorname{\widetilde Hess}E(\bm x^\ast)[\bm{v}_{i}^\ast], \bm{v}_i^\ast\right\rangle$.
We now show by induction that for $i=1,\cdots,k$,
\begin{equation}\label{eqn:vimui}
\operatorname{\widetilde Hess}E(\bm x^\ast)[\bm{v}_{i}^\ast]= \mu_i^\ast \bm{v}_i^\ast\neq \bm 0,\quad
\left\langle \bm{v}_j^\ast, \bm{v}_i^\ast \right\rangle=\delta_{ij}, \quad j=1,\cdots,i-1.
\end{equation}
The $i=1$ case is obtained from \eqref{eqn:vidotequ0} directly, and $\bm v_1^\ast\in T(\bm x^\ast)$ indicates $\mu_1\neq 0$ due to the nondegeneracy.
Assumed that \eqref{eqn:vimui} holds for $1\leqslant i< l$, by taking $i=l$ in \eqref{eqn:vidotequ0} we have,
\begin{equation}\label{eqn:viin1}
\left(\operatorname{\widetilde Hess}E(\bm x^\ast) - \sum\limits_{j=1}^{l-1} 2\mu_j^\ast\bm{v}_j^\ast{\bm{v}_j^\ast}^\top\right) \bm{v}_{l}^\ast= \mu_l^\ast \bm{v}_l^\ast,
\end{equation}
from the symmetry of $\operatorname{\widetilde Hess}E(\bm x^\ast)$.
Since $\bm{v}_1^\ast,\cdots,\bm{v}_{l-1}^\ast$ are eigenvectors of $\operatorname{\widetilde Hess}E(\bm x^\ast)$ according to the inductive assumption, $\operatorname{\widetilde Hess}E(\bm x^\ast)$ and $\operatorname{\widetilde Hess}E(\bm x^\ast) - \sum\limits_{j=1}^{l-1} 2\mu_j^\ast\bm{v}_j^\ast{\bm{v}_j^\ast}^\top$ share the same eigenvectors, so $\bm{v}_l^\ast\in T(\bm x^\ast)$ is also an eigenvector of $\operatorname{\widetilde Hess}E(\bm x^\ast)$ with an eigenvalue $\mu_l^\ast\neq 0$.
Furthermore, $\sum\limits_{j=1}^{l-1} \mu_j^\ast \langle\bm{v}_l^\ast,\bm{v}_{j}^\ast\rangle\bm{v}_j^\ast=\bm 0$ leads to $\langle\bm{v}_l^\ast,\bm{v}_{j}^\ast\rangle=0$, which completes the induction of \eqref{eqn:vimui}.
Consequently, from $\dot{\bm x}=\bm 0$ and $\bm x^\ast \in\mathcal{M}$, we have $\operatorname{grad} E(\bm{x^\ast})=\bm 0$ and the Jacobian \eqref{eqn:chisdmjacobian} is block lower triangular.

Finally, we show that the index of the stationary point $\bm x^\ast$ is $k$, and $\mu_i^\ast=\lambda_i^\ast$ for $i=1,\cdots,k$.
Since $\operatorname{Hess}E(\bm x^\ast)$ is a symmetric operator on $T(\bm x^\ast)$ with some eigenpairs $\{(\mu_i^\ast,\bm v_i^\ast)\}_{i=1}^k$, we denote $\bm v_i^\ast$ $(i=k+1,\cdots,d-m)$ as the eigenvector of its eigenvalue $\mu_i^\ast$ such that $\{\bm v_i^\ast\}_{i=1}^{d-m}$ is an orthonormal basis of $T(\bm x^\ast)$.
Similarly to \eqref{eqn:chisdjx1} and \eqref{eqn:chisdjx2}, the eigenvalues of $\mathbb{J}_{\bm x}^\ast$ are
\begin{equation}\label{eqn:jxeigenvalues}
\mu_1^\ast,\cdots,\mu_k^\ast;\quad
-\mu_{k+1}^\ast,\cdots,-\mu_{d-m}^\ast;\quad
-\mu \|{\nabla c_1(\bm x^\ast)}\|^2, \cdots, -\mu \|{\nabla c_m(\bm x^\ast)}\|^2;
\end{equation}
which have negative real parts from the linear stability.
Therefore, $\{\mu_i^\ast\}_{i=1}^{k}$ are negative and $\{\mu_i^\ast\}_{i=k+1}^{d-m}$ are positive, so $\bm x^\ast$ is a $k$-saddle.
Similarly to \eqref{eqn:chisdji1} and \eqref{eqn:chisdji2}, the eigenvalues of $\mathbb{J}_{i}^\ast$,
\begin{equation}\label{eqn:jieigenvalues}
\mu_i^\ast+\mu_1^\ast, \cdots, \mu_i^\ast+\mu_i^\ast;\quad
\mu_i^\ast-\mu_{i+1}^\ast, \cdots, \mu_i^\ast-\mu_{d-m}^\ast;\quad
\mu_i^\ast (\text{with multiplicity }m);
\end{equation}
have negative real parts as well, which indicates $\mu_i^\ast<\mu_{i+1}^\ast$ for $i=1,\cdots,k$.
Comparing the negative eigenvalues of $\operatorname{Hess}E(\bm x^\ast)$, we have $\mu_i^\ast=\lambda_i^\ast$, which completes our proof.
\end{proof}

\begin{remark}\label{rmk:identity}
The CHiSD \eqref{eqn:chisd} actually searches for $k$-saddles on the isosurface of $\bm c(\bm x)$, while the initial condition \eqref{eqn:chisdinitial} makes the dynamics search for saddle points on the manifold $\mathcal{M}$.
The additional term in \eqref{eqn:chisdm} is applied to pull the dynamics towards the manifold in case of perturbation in order to obtain linear stability.

From Remark \ref{rmk:consercation}, the dynamics \eqref{eqn:chisdm} and the dynamics \eqref{eqn:chisd} with the initial condition \eqref{eqn:chisdinitial} have the same orbits.
Therefore, a linearly stable steady state $(\bm x^\ast,\bm v_1^\ast,\cdots,\bm v_k^\ast)$ of \eqref{eqn:chisdm} is also asymptotically stable for \eqref{eqn:chisd} in the following sense.
As long as the initial condition \eqref{eqn:chisdinitial} is sufficiently close to $(\bm x^\ast,\bm v_1^\ast,\cdots,\bm v_k^\ast)$, the dynamics \eqref{eqn:chisd} will converge to this steady state as well.
\end{remark}

\section{Numerical implementation}
\label{sec:numericalimplementation}
\subsection{Retractions and vector transport}
\label{sec:retractionstransports}
The numerical scheme of CHiSD \eqref{eqn:chisd} is supposed to maintain the manifold property \eqref{eqn:manifoldproperty} in each iteration step.
Since a simple explicit Euler scheme can easily push $\bm x$ away from the manifold, we introduce the tools of retractions and vector transport in manifold optimization to discretize CHiSD \eqref{eqn:chisd}, and we refer to \cite{absil2008optimization} for more detailed information.

Strictly speaking, a tangent vector $\bm \xi_{\bm x}$ to the manifold $\mathcal{M}$ at a point $\bm x\in \mathcal{M}$ is a mapping from the set of germs of functions $\mathfrak{F}_{\bm x}(\mathcal{M})$ to $\mathbb{R}$.
Each tangent vector $\bm \xi_{\bm x}$ can be represented by a curve $\gamma$ on $\mathcal{M}$ with $\gamma(0)=\bm x$ satisfying $\bm \xi_{\bm x}f=\frac{\mathrm{d}}{\mathrm{d}t}f(\gamma(t))\big|_{t=0}$ for $\forall f \in \mathfrak{F}_{\bm x}(\mathcal{M})$.
The tangent space at $\bm x\in \mathcal{M}$, denoted by $T_{\bm x}\mathcal{M}$, can now be canonically identified with $T(\bm x)$ via the one-to-one correspondence $\bm \xi_{\bm x} \in T_{\bm x}\mathcal{M}\mapsto \gamma'(0)=\lim\limits_{\tau\to0}\frac{\gamma(\tau)-\gamma(0)}{\tau} \in T(\bm x)$.
Therefore, for $\bm x\in \mathcal{M}$, we treat the mapping $\bm \xi_{\bm x} \in T_{\bm x}\mathcal{M}$ and the tangent vector object $\bm \xi\in T(\bm x)$ at $\bm x$ equally throughout the article.
The tangent bundle $T\mathcal{M}:=\bigcup\limits_{\bm x\in\mathcal{M}}T_{\bm x}\mathcal{M}$ denotes the set of all tangent vectors to $\mathcal{M}$.

The notion of moving $\bm x$ on the manifold in the direction of a tangent vector is generalized by a retraction mapping.
A retraction $R$ is a smooth mapping from $T\mathcal{M}$ to $\mathcal{M}$, and the retraction of $R$ to $T_{\bm x}\mathcal{M}$,
denoted by $R_{\bm x}$, satisfies
$R_{\bm x}(\bm 0_{\bm x})={\bm x}$ and
$\frac{\mathrm{d}}{\mathrm{d}t} R_{\bm x}\left(t\bm \eta_{\bm x}\right)\big|_{t=0} =\bm \eta_{\bm x}$ for $\forall \bm \eta_{\bm x} \in T_{\bm x}\mathcal{M}$, where $\bm 0_{\bm x}$ is the zero element in $T_{\bm x}\mathcal{M}$.
As a natural retraction on the Riemannian manifold in the geometric sense, the exponential mapping at $\bm x \in \mathcal{M}$, denoted by $\operatorname{Exp}_{\bm x}$, maps $\bm\eta\in T_{\bm x}\mathcal{M}$ to $\gamma(1;\bm x,\bm \eta)\in \mathcal{M}$, where $\gamma(t;\bm x,\bm \eta)$ is the unique geodesic such that $\gamma(0)=\bm x$ and $\dot\gamma(0)=\bm \eta$.
For the unit sphere $S^{d-1}$, the exponential mapping has a computable form,
\begin{equation}\label{eqn:exponentialmappingsphere}
\operatorname{Exp}_{\bm x}\bm\eta= (\cos \|\bm\eta\|) \bm x+\dfrac{\sin\|\bm\eta\|}{\|\bm\eta\|}\bm\eta,
\end{equation}
while the exponential mapping of a general Riemannian manifold poses significant numerical challenges to calculating cheaply.
Therefore, other computable retractions are often applied as approximations of the exponential mapping.
It should be pointed out that for the same manifold, we can have different choices of retractions, which might lead to different computational results.
For the unit sphere $S^{d-1}$, another clear form of the retraction is
\begin{equation}\label{eqn:retractionsphere}
R_{\bm x}(\bm \eta)=\dfrac{\bm x+\bm \eta}{\|\bm x+\bm \eta\|}
=\operatorname{Exp}_{\bm x}\left(\dfrac{\arctan \|\bm\eta\|}{\|\bm\eta\|}\bm\eta\right).
\end{equation}

Next we consider enforcing the directions $\bm v_i$ on the tangent space.
The notion of vector transport $\mathcal{T}$ specifies how to transport a tangent vector $\bm \xi_{\bm x}\in T_{\bm x}\mathcal{M}$ as $\bm x$ moves on the Riemannian manifold.
A vector transport on a manifold $\mathcal{M}$ is a smooth mapping from the Whitney sum $T\mathcal{M}\oplus T\mathcal{M}:=
\{(\bm\eta_{\bm x}, \bm \xi_{\bm x}): \bm\eta_{\bm x}, \bm \xi_{\bm x}\in T_{\bm x}\mathcal{M}, \bm x\in\mathcal{M}\}$ to $T\mathcal{M}$ satisfying following properties:
\begin{enumerate}
  \item $\mathcal{T}_{\bm \eta_{\bm x}}\bm \xi_{\bm x}\in T_{R_{\bm x}(\bm \eta_{\bm x})}\mathcal{M}$, where $R$ is the retraction associated with $\mathcal{T}$;
  \item $\mathcal{T}_{\bm 0_{\bm x}}\bm \xi_{\bm x}=\bm \xi_{\bm x}$ for $\forall\bm \xi_{\bm x} \in T_{\bm x}\mathcal{M}$;
  \item $\mathcal{T}_{\bm \eta_{\bm x}}(a\bm \xi_{\bm x}+b\bm \zeta_{\bm x})=
  a\mathcal{T}_{\bm \eta_{\bm x}}\bm \xi_{\bm x} +b\mathcal{T}_{\bm \eta_{\bm x}}\bm \zeta_{\bm x}$
  for $\forall\bm \xi_{\bm x},\bm \zeta_{\bm x} \in T_{\bm x}\mathcal{M}$ and $\forall a,b\in \mathbb{R}$.
\end{enumerate}
The parallel translation \eqref{eqn:paralleltranport} is a natural vector transport along a curve $\bm x(t)$ on the manifold, where the parallel translation of orthonormal vectors remains orthonormal.
For the unit sphere $S^{d-1}$, as $\bm x \in S^{d-1}$ moves to $\operatorname{Exp}_{\bm x}\bm\eta$ along the geodesics, the parallel translation of a vector $\bm\xi\in T_{\bm x}\mathcal{M}$ has a closed form,
\begin{equation}\label{eqn:paralleltranslationsphere}
\mathcal{T}_{\bm \eta}\bm \xi=
\bm \xi +\dfrac{\cos\|\bm\eta\|-1}{\|\bm\eta\|^2}\langle\bm \eta, \bm \xi\rangle \bm\eta
-\dfrac{\sin\|\bm\eta\|}{\|\bm\eta\|}\langle\bm \eta, \bm \xi\rangle\bm x
\in T_{\operatorname{Exp}_{\bm x}\bm\eta}\mathcal{M},
\end{equation}
where the exponential mapping \eqref{eqn:exponentialmappingsphere} is the associated retraction.
Parallel translation is not the only way to achieve vector transport, and is often difficult to calculate numerically for general Riemannian manifolds as well.
Alternatively, there is considerable flexibility in how to choose the vector transport, and two typical approaches to generating computationally tractable vector transport according to the retraction operator $R$ are
\begin{equation}\label{eqn:vectortransport}
\mathcal{T}_{\bm \eta}\bm \xi=\frac{\mathrm{d}}{\mathrm{d}t} R_{\bm x}(\bm\eta+t\bm\xi)\bigg|_{t=0}, \text{ and }
\mathcal{T}_{\bm \eta}\bm \xi=\mathbf{P}_{T(R_{\bm x}(\eta))}\bm \xi.
\end{equation}
Note that vector transport generally may not maintain the orthonormality of vectors.
For the unit sphere $S^{d-1}$ with the retraction \eqref{eqn:retractionsphere}, the two approaches yield
\begin{equation}\label{eqn:vectortransportsphere}
\mathcal{T}_{\bm \eta}\bm \xi=\dfrac{1}{\|\bm x+\bm \eta\|}\bm \xi-\dfrac{\langle\bm\xi, \bm x+\bm \eta\rangle}{\|\bm x+\bm \eta\|^3}(\bm x+\bm \eta),  \text{ and }
\mathcal{T}_{\bm \eta}\bm \xi=\bm \xi-\dfrac{\langle\bm\xi, \bm x+\bm \eta\rangle}{\|\bm x+\bm \eta\|^2}(\bm x+\bm \eta).
\end{equation}

\subsection{Numerical algorithms of CHiSD}
\label{sec:numericalalgorithmCHiSD}
With the help of vector transport $\mathcal{T}$ and the retraction $R$ associated with $\mathcal{T}$, we are able to implement CHiSD \eqref{eqn:chisd} numerically with the initial condition \eqref{eqn:hisdinitial}.
Specifically, at the $(n+1)$-th iteration step, we aim to calculate $\bm x^{(n+1)}$ and $\bm v_i^{(n+1)}$ using $\bm x^{(n)}$ and $\bm v_i^{(n)}$ in the previous step.
The retraction $R_{\bm x}$ provides a practical way to pull the points from $\bm x+T(\bm x)$ back onto the manifold $\mathcal{M}$, so we can implement an explicit scheme with a retraction as
\begin{equation}\label{eqn:xiteration}
\bm x^{(n+1)}=R_{\bm x^{(n)}}\left(\alpha^{(n)} \bm g^{(n)}\right), \quad
\bm g^{(n)}=-\left(\mathbf{I}-\sum_{i=1}^{k} 2 \bm{v}_{i}^{(n)} {\bm{v}_{i}^{(n)}}^\top\right)
\operatorname{grad} E\left(\boldsymbol{x}^{(n)}\right),
\end{equation}
to calculate $\bm x^{(n+1)}$.
Note that $\bm v_1^{(n)},\cdots,\bm v_k^{(n)}$ are orthonormal vectors in $T(\bm x^{(n)})$ which approximate the corresponding eigenvectors of $\operatorname{Hess}E(\bm x^{(n)})$, and $\alpha^{(n)}$ is the step size to be determined.
Therefore, $\bm g^{(n)}$ lies in the tangent space $T(\bm x^{(n)})$, and this iteration scheme ensures that $\bm x^{(n)}$ always satisfies the equality constraints \eqref{eqn:constraints}.

The dynamics of $\bm v_i$ consists of two terms: the first term solves the eigenvalue problem \eqref{eqn:chisdoptviwidetilde} using gradient flow, while the second one corresponds to the parallel translation as discussed in Remark \ref{rmk:parallel}.
Accordingly, we deal with the two terms successively.
We first transport the previous vectors $\{\bm v_i^{(n)}\}$ to $\{\check{\bm v}_i^{(n)}\}$ in the tangent space $T(\bm x^{(n+1)})$ using the vector transport $\mathcal{T}_{\alpha^{(n)} \bm g^{(n)}}$ at $\bm x^{(n)}$, and then solve the eigenvalue problem \eqref{eqn:chisdoptviwidetilde}.
Although an exact solution to the eigenvalue problem works well, we recommend to find a rough solution to reduce the computational costs, and $\bm v_i$ will converge as $\bm x$ approaches a $k$-saddle.
Since $\bm x^{(n+1)}$ is close to $\bm x^{(n)}$, the transported vector $\check{\bm v}_i^{(n)}$ provides a good initial guess of the corresponding eigenvector, so we simply apply one-step gradient flow for each eigenvector.
Practically, the Riemannian Hessian $\operatorname{Hess}E(\bm x)$ is often expensive to calculate, so a dimer approximation $\bm u_i^{(n)}$ is applied to approximate $\operatorname{\widetilde Hess}E(\bm x^{(n+1)})[\check{\bm v}_i^{(n)}]$ with a small $l>0$ \cite{henkelman1999dimer}.
Finally, $\left\{\bm v_i^{(n+1)}\right\}$ is calculated from $\left\{\tilde{\bm v}_i^{(n+1)}\right\}$ using a normalized orthogonalization function $\operatorname{orth}(\cdot)$ which can be realized by a Gram--Schmidt procedure.
Consequently, $\{\bm v_i^{(n+1)}\}$ can be calculated with the following procedure in each iteration step,
\begin{equation}\label{eqn:viteration}
\left\{
\begin{aligned}
&\check{\bm v}_i^{(n)} =\mathcal{T}_{\alpha^{(n)} \bm g^{(n)}} \bm v_i^{(n)},\\
&\bm u_i^{(n)}=\mathbf{P}_{T(\bm x^{(n+1)})}\dfrac{\operatorname{grad}E(\bm x^{(n+1)}+l \check{\bm v}_i^{(n)}) - \operatorname{grad}E(\bm x^{(n+1)}-l \check{\bm v}_i^{(n)} )}{2l},\\
&\bm d_i^{(n)}=-\bm u_i^{(n)}+\left\langle \bm u_i^{(n)}, \check{\bm v}_i^{(n)}\right\rangle \check{\bm v}_i^{(n)}
+\sum\limits_{j=1}^{i-1} 2\left\langle \bm u_i^{(n)}, \check{\bm v}_j^{(n)}\right\rangle \check{\bm v}_j^{(n)},\\
&\tilde{\bm v}_i^{(n+1)}=\check{\bm v}_i^{(n)}+\beta^{(n)}\bm d_i^{(n)},\\
&\left[\bm v_1^{(n+1)},\cdots,\bm v_k^{(n+1)}\right]=\operatorname{orth}
\left(\left[\tilde{\bm v}_1^{(n+1)},\cdots,\tilde{\bm v}_k^{(n+1)}\right]\right).
\end{aligned}
\right.
\end{equation}
The iteration \eqref{eqn:xiteration}--\eqref{eqn:viteration} is terminated if $\|\operatorname{grad}E(\bm x^{(n)})\|$ is smaller than the tolerance.

\begin{remark}\label{rmk:lobpcg}
The above iteration \eqref{eqn:viteration} for $\bm v_i$ can be regarded as a one-step application of the power method for $\mathbf{I}-\beta^{(n)}\operatorname{\tilde Hess}E(\bm x^{(n+1)})$ on the tangent space $T(\bm x^{(n+1)})$.
\end{remark}

\subsection{Construction of the solution landscape}
\label{sec:solutionlandscape}
In this subsection, we introduce a systematic numerical procedure to construct the solution landscape of an energy functional subject to equality constraints.
This procedure consists of a downward search algorithm and an upward search algorithm, which is a generalization of the pathway map to constrained cases \cite{yin2020construction}.

First, we present a toy model to illustrate our basic idea.
Consider an energy functional constrained on $S^2$,
\begin{equation}\label{eqn:toy}
\begin{aligned}
  E(x_1, x_2, x_3) &= (x_1^2-1)^2 + x_2^2 + 2 x_3^2,\\
  \mathrm{s.t.} \quad c(x_1, x_2, x_3) &= x_1^2+x_2^2+x_3^2-1=0,
\end{aligned}
\end{equation}
whose energy landscape is shown in Fig.~\ref{fig:ill}(a).
Two minima $C_1(1,0,0)$ and $C_2(-1,0,0)$ are connected by two 1-saddles $B_1(0,1,0)$ and $B_2(0,-1,0)$.
We aim to find these minima and 1-saddles down from an index-2 critical point $A(0,0,1)$, and construct the solution landscape as shown in Fig.~\ref{fig:ill}(b), which depicts how lower-index saddle points are connected to the higher-index ones.
The arrows in Fig.~\ref{fig:ill}(b) are realized using downward search algorithms based on the CHiSD.
\begin{figure}[htbp]
\centering
\includegraphics[width=.8\linewidth]{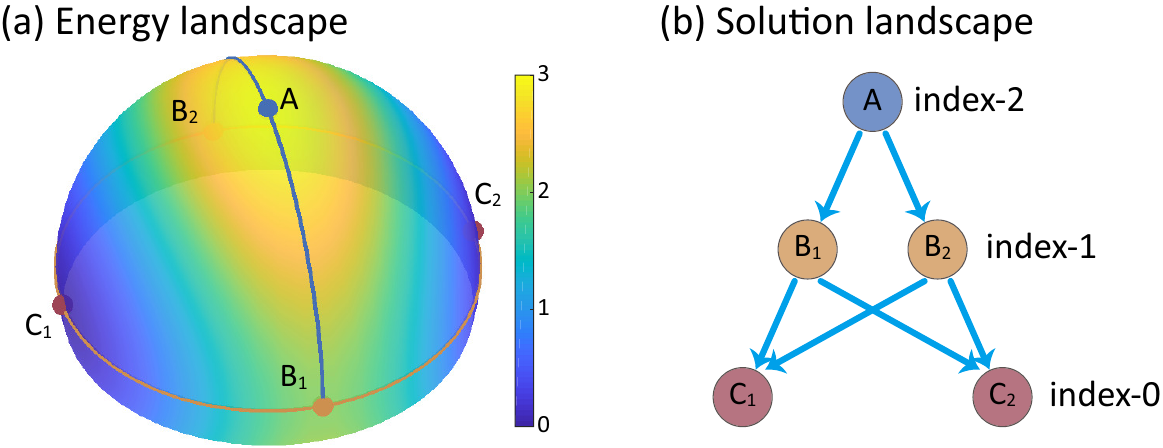}
\caption{Illustration of (a) an energy landscape of the function \eqref{eqn:toy} and (b) the solution landscape starting from the maximum, $A$, down to minima $C_1$ and $C_2$.
Two saddle points $B_1$ and $B_2$ are connected to $A$.
The indices labeled in (b) are those according to the Morse definition.}
\label{fig:ill}
\end{figure}

Starting from a parent state (high-index saddle point), the downward search is the core procedure to search for the stationary points with lower indices that are connected to this parent state.
Given a $k$-saddle $\hat{\bm x}\in\mathcal{M}$, we let $\hat{\bm v}_1,\cdots,\hat{\bm v}_k\in T(\hat{\bm x})$ denote the $k$ orthonormal eigenvectors of $\operatorname{Hess}E(\hat{\bm x})$ with negative eigenvalues $\hat{\lambda}_1 \leqslant \cdots \leqslant \hat{\lambda}_k<0$.
First we slightly perturb $\hat{\bm x}$ along an unstable direction $\hat{\bm v}_j$ chosen from the unstable directions $\{\hat{\bm v}_1,\cdots,\hat{\bm v}_k\}$ with $0<\varepsilon \ll 1$.
Then an $m$-CHiSD $(m<k)$ is started from the point $R_{\hat{\bm x}}\left(\pm \varepsilon \hat{\bm v}_j\right)$ and the $m$ initial directions $\left\{\bm v_i(0)\right\}_{i=1}^m$ are chosen from the unstable directions excluding $\hat{\bm v}_j$.
A typical choice of the initial condition is $\left(R_{\hat{\bm x}}\left( \pm\varepsilon \hat{\bm v}_{m+1}\right), \hat{\bm v}_1,\cdots, \hat{\bm v}_{m}\right)$.
Once a new $m$-saddle is found, we implement another downward search from this newly-found saddle recursively until minima are found.
The downward search algorithm is presented in detail as Algorithm~\ref{alg:downward}.
Note that each element $(\bm x,m,\{\bm v_1,\cdots,\bm v_k\})$ in the queue $\mathcal{A}$ represents to find $m$-saddles from the $k$-saddle $\bm x$.

\begin{algorithm}[htb]\caption{Downward search}\label{alg:downward}
\begin{algorithmic}[1]
\REQUIRE{A $\hat{k}$-saddle $\hat{\bm x}$, $\varepsilon>0$.}
\STATE{Calculate the $\hat{k}$ eigenvectors $\hat{\bm v}_1,\cdots,\hat{\bm v}_{\hat{k}}\in T(\hat{\bm x})$ of $\operatorname{Hess}E(\hat{\bm x})$;}
\STATE{Set the queue $\mathcal{A}= \{(\hat{\bm x},\hat{k}-1,\{\hat{\bm v}_1,\cdots,\hat{\bm v}_{\hat{k}}\})\}$,
the solution set $\mathcal{S}=\{\hat{\bm x}\}$, and the relation set $\mathcal{R}=\varnothing$;}
\WHILE{$\mathcal{A}$ is not empty}
  \STATE{Pop $(\bm x,m,\{\bm v_1,\cdots,\bm v_k\})$ from $\mathcal{A}$;}
  \STATE{Push $(\bm x,m-1,\{\bm v_1,\cdots,\bm v_k\})$ into $\mathcal{A}$ if $m\geqslant1$;}
  \FOR{$j=1:k$}
   \STATE{Determine the initial directions: $\{\bm v_i: i=1, \cdots, m+1, i\neq \min(j,m+1)\}$;}
    \IF{$m$-CHiSD from $R_{\bm x}\left(\pm \varepsilon \bm v_j\right)$ converges to $(\widetilde{\bm x}, \widetilde{\bm v}_1, \cdots, \widetilde{\bm v}_m)$}
      \STATE{$\mathcal{R}\leftarrow\mathcal{R} \cup \{(\bm x,\widetilde{\bm x})\}$;}
      \IF{$\widetilde{\bm x} \notin \mathcal{S}$}
        \STATE{$\mathcal{S}\leftarrow\mathcal{S} \cup \{\widetilde{\bm x}\}$;}
        \STATE{Push $(\widetilde{\bm x},m-1,\{\widetilde{\bm v}_1,\cdots,\widetilde{\bm v}_{m}\})$ into $\mathcal{A}$ if $m \geqslant 1$;}
      \ENDIF
    \ENDIF
  \ENDFOR
\ENDWHILE
\ENSURE{The solution set $\mathcal{S}$ and the relation set $\mathcal{R}$.}
\end{algorithmic}
\end{algorithm}

With the downward search, we can systematically search for multiple stationary points from a given parent state in a controlled procedure.
However, if the parent state is unknown, or in the case that multiple parent states exist, we need a numerical procedure to find a parent state.
Therefore, an upward search algorithm is required to find a higher-index saddle from a low-index stationary point (usually a minimum).
Fortunately, the CHiSD also embeds a mechanism to search upward.
Given a $k$-saddle $\hat{\bm x}\in\mathcal{M}$, more eigenvectors of $\operatorname{Hess}E(\hat{\bm x})$ with smallest eigenvalues are calculated as $\hat{\bm v}_1,\cdots,\hat{\bm v}_K\in T(\bm x)$, where $K$ $(K>k)$ is the highest index of the saddle point to search for.
Different from the downward search, we choose a direction $\hat{\bm v}_j$ from the stable directions $\{\hat{\bm v}_{k+1},\cdots,\hat{\bm v}_{K}\}$, and slightly perturb $\hat{\bm x}$ along $\hat{\bm v}_j$.
An $m$-CHiSD $(m>k)$ procedure is then started from $R_{\hat{\bm x}}\left(\pm \varepsilon \hat{\bm v}_j\right)$, while the $m$ initial directions $\{\bm v_i(0)\}_{i=1}^m$ need to include $\hat{\bm v}_j$.
Then this procedure is implemented to newly-found saddles recursively until $K$-saddles are found or no higher-index saddles can be found.
A typical choice of the initial condition is $(R_{\hat{\bm x}}\left(\pm \varepsilon \hat{\bm v}_m\right), \hat{\bm v}_1,\cdots, \hat{\bm v}_m)$.
Besides the difference in initial conditions, the upward search only aims to find parent states of the solution landscape, while the downward search presents the relations between stationary points and gives a complete picture of the solution landscape by exhaustively searching for multiple stationary points.
Algorithm~\ref{alg:upward} presents the upward search algorithm in detail with the typical choice.

\begin{algorithm}[htb]\caption{Upward search}\label{alg:upward}
\begin{algorithmic}[1]
\REQUIRE{A $\hat{k}$-saddle $\hat{\bm x}$, $\varepsilon>0$, the highest index $K$.}
\STATE{Calculate the $\hat{k}$ eigenvectors $\hat{\bm v}_1,\cdots,\hat{\bm v}_{\hat{K}}\in T(\hat{\bm x})$ of $\operatorname{Hess}E(\hat{\bm x})$;}
\STATE{Set the stack $\mathcal{A}= \{(\hat{\bm x},\hat{k}+1,\{\hat{\bm v}_1,\cdots,\hat{\bm v}_{K}\})\}$ and the solution set $\mathcal{S}=\{\hat{\bm x}\}$,}
\WHILE{$\mathcal{A}$ is not empty}
  \STATE{Pop $(\bm x,m,\{\bm v_1,\cdots,\bm v_K\})$ from $\mathcal{A}$;}
  \STATE{Push $(\bm x,m+1,\{\bm v_1,\cdots,\bm v_K\})$ into $\mathcal{A}$ if $m<K$;}
  \IF{$m$-CHiSD from $(R_{\bm x}(\pm \varepsilon \bm v_m), \bm v_1,\cdots,\bm v_m)$ converges to $(\widetilde{\bm x},\widetilde{\bm v}_1,\cdots,\widetilde{\bm v}_{m})$}
    \IF{$\widetilde{\bm x} \notin \mathcal{S}$}
      \STATE{$\mathcal{S}\leftarrow\mathcal{S} \cup \{\widetilde{\bm x}\}$ and calculate more eigenvectors $\widetilde{\bm v}_{m+1},\cdots,\widetilde{\bm v}_{\widetilde{K}}$ of $\operatorname{Hess}E(\widetilde{\bm x})$;}
      \STATE{Push $(\widetilde{\bm x},m+1,\{\widetilde{\bm v}_1,\cdots,\widetilde{\bm v}_K\})$ into $\mathcal{A}$ if $m<K$;}
    \ENDIF
  \ENDIF
\ENDWHILE
\ENSURE{The solution set $\mathcal{S}$.}
\end{algorithmic}
\end{algorithm}

The computation cost of downward and upward search mainly depends on the CHiSD method, since each process can implement the CHiSD search individually, the algorithm can be naturally parallelized to speed it up.
In practice, the solution landscape is achieved by a combination of downward search and upward search, and we can navigate up and down systematically on the energy landscape to construct a complete solution landscape.

\section{Numerical examples}
\label{sec:numericalexamples}
\subsection{Thomson problem}
\label{sec:thomson}
The Thomson problem considers the minimal-energy configuration of $N$ classical charged particles confined to a sphere which interact with each other via a Coulomb potential $f(r)=r^{-1}$, and was originally proposed as a representation of the atomic structure \cite{thomson1904structure}.
The Thomson problem has attached much attention as a special case of the 7th problem in Steven Smale’s eighteen problems for the 21st century \cite{smale1998mathematical}.
Many numerical attempts have been made to find local and global minima \cite{altschuler1994method,erber1991equilibrium}, transition states \cite{mehta2016kinetic,zhang2012shrinking} and high-index saddle points \cite{mehta2015exploring} to fully understand the energy landscape of the Thomson problem.

In the Thomson problem, the coordinate of the $i$-th particle $\bm x_i=(x_i, y_i, z_i)\in \mathbb{R}^3$ is constrained on the unit sphere $S^2$, and the energy function of a configuration $\left(\bm{x}_1, \cdots, \bm{x}_N\right) \in\mathbb{R}^{3\times N}$ is,
\begin{equation}\label{eqn:thomson}
E(\bm{x}_1, \cdots, \bm{x}_N)=\sum_{i<j}f\left(\left\|\bm{x}_{i}-\bm{x}_{j}\right\|\right).
\end{equation}
Because of the rotation symmetry, we specify the first particle to be the north pole and the second particle to lie on the $yz$-plane \cite{zhang2012shrinking}.
In other words, the state variables are constrained on a manifold,
\begin{equation}\label{eqn:thomsonlinear}
\mathcal{M} = \left\{(\bm{x}_1, \cdots, \bm{x}_N)\in \mathbb{R}^{3\times N}:
\bm x_1 = (0,0,1),x_2 = 0,\|\bm x_i\|_2=1 \right\},
\end{equation}
so that the Riemannian Hessian at a stationary point has no zero eigenvalues generally.
For $\bm x = (\bm{x}_1, \cdots, \bm{x}_N) \in \mathcal{M}$ and
$\bm \eta = \left(\bm 0, (0,\bm \eta_2),\bm \eta_3, \cdots, \bm \eta_N\right)\in T_{\bm x}\mathcal{M}$ where $\bm \eta_2 \in T_{(y_2, z_2)}S^1$ and $\bm\eta_i\in T_{\bm x_i}S^2$ for $i \geqslant 3$, the retraction operator is,
\begin{equation}\label{eqn:thomsonretraction}
  R_{\bm x}\bm \eta = \left(\bm x_1, (x_2, R^{S^1}_{(y_2, z_2)}\bm \eta_2),
  R^{S^2}_{\bm x_3}\bm \eta_3,\cdots,R^{S^2}_{\bm x_N}\bm\eta_N\right),
\end{equation}
where $R^{S^{d-1}}$ is the retraction operator on the unit sphere $S^{d-1}$ chosen as the exponential mapping \eqref{eqn:exponentialmappingsphere}.
For $\bm \xi = \left(\bm 0, (0,\bm \xi_2),\bm \xi_3, \cdots, \bm \xi_N\right)\in T_{\bm x}\mathcal{M}$, the vector transport from $\bm x$ to $R_{\bm x}\bm \eta$ is
\begin{equation}\label{eqn:thomsonvectortransport}
  \mathcal{T}_{\bm \eta}\bm \xi = \left(\bm 0, (0, \mathcal{T}^{S^1}_{\bm \eta_2}\bm \xi_2),
  \mathcal{T}^{S^2}_{\bm \eta_3}\bm \xi_3,\cdots,
  \mathcal{T}^{S^2}_{\bm \eta_N}\bm \xi_N\right),
\end{equation}
where $\mathcal{T}^{S^{d-1}}$ is the vector transport for the unit sphere $S^{d-1}$ with the form of \eqref{eqn:paralleltranslationsphere}.
Note that if $N$ is large, the constraint amount $m$ will be about as large as $N$.
Nevertheless, the computational costs of the retraction operator \eqref{eqn:thomsonretraction} and the vector transport \eqref{eqn:thomsonvectortransport} will be $\mathcal{O}(N)$ because of the product structure of \eqref{eqn:thomsonlinear}.

\begin{figure}[htbp]
  \centering
  \includegraphics[width=0.8\linewidth]{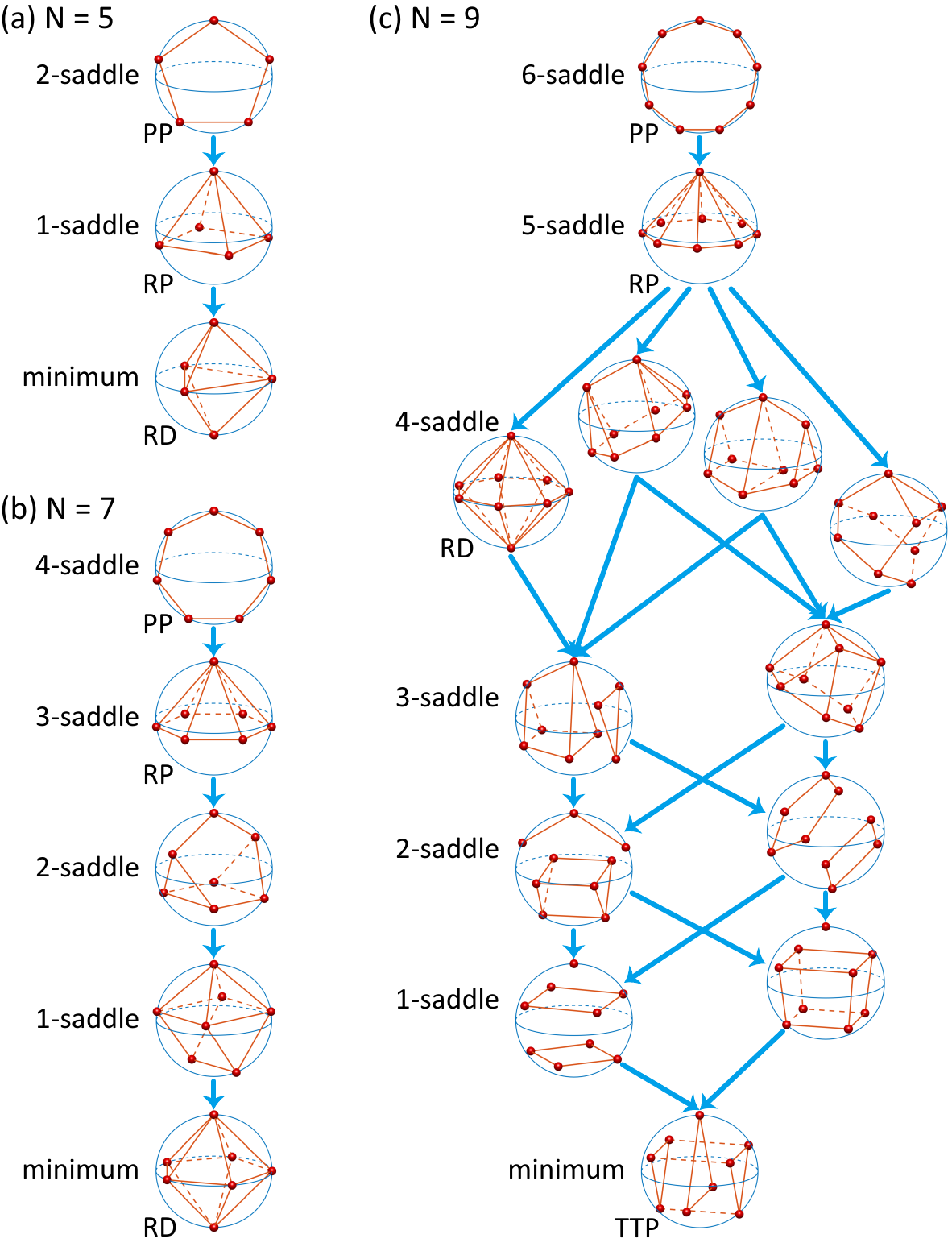}
  \caption{The solution landscape of the Thomson problem with particle numbers (a) $N=5$. (b) $7$, and (c) $9$.
  The configurations of stationary points are presented in the solution landscape.
  Each red ball represents a particle on the sphere, and the red line segments are drawn to show the relative positions.
  The height of each configuration approximately corresponds to its relative energy.
  The index of each stationary point is labelled on the left side, and some important stationary points are further labelled with their configurations.
  Each arrow from a higher-index stationary point to a lower-index stationary point corresponds to a CHiSD pathway by the downward search.
  }\label{fig:thomson}
\end{figure}

We aim to present multiple stationary points of the Thomson problem in order to offer a full description for the energy landscape without random initial guesses.
A trivial stationary point of the Thomson problem is a planar polygon configuration (PP) that $N$ particles are evenly spaced on a great circle of $S^2$, and is an $(N-3)$-saddle for $N\geqslant 3$.
This trivial stationary point presents a natural parent state for downward search which is implemented with fixed step sizes $\alpha^{(n)}=10^{-4}$ and $\beta^{(n)}=10^{-3}$.
In this low-dimensional example, we adopt the canonical inner product and the Hessian is evaluated exactly.
The computational results for some small particle numbers $N=5,7$ and $9$ are presented in Fig.~\ref{fig:thomson}, where the stationary points with the same configuration are only shown once.
From the trivial $(N-3)$-saddle, an $(N-4)$-saddle of a regular pyramid configuration (RP) is first located during the downward search, where $(N-1)$ particles evenly spaced on a latitude ring.
The regular dipyramid configuration (RD) with two antipodal particles and $(N-2)$ other particles evenly spaced on the equator is the minimal-energy configuration for $N=5,7$, and a 4-saddle configuration for $N=9$, while the minimal-energy configuration for $N=9$ is a triaugmented triangular prism (TTP).

\subsection{Bose--Einstein condensation}
\label{sec:bec}
The experimental realization of BEC in vapors of magnetically-trapped alkali atoms that occupy a single quantum state has attracted great interest in the atomic physics community \cite{anderson1995observation,davis1995bose},
The properties of BEC at an ultra-low temperature are well described by the macroscopic wave function $\psi(\bm x,t)$, whose evolution is governed by the nonlinear Schr\"{o}dinger equation (Gross--Pitaevskii equation),
\begin{equation}\label{eqn:gpe}
\mathrm{i}\dfrac{\partial}{\partial t}\psi(\bm x,t)=\left(-\dfrac12 \nabla^2+V(\bm x)+ \beta|\psi(\bm x,t)|^2\right) \psi(\bm x,t),\quad \bm x\in \mathbb{R}^{d},
\end{equation}
where $V(\bm x)$ is a real-valued trapping potential on $\mathbb{R}^{d}$ and $\beta$ is the dimensionless interaction coefficient \cite{bao2013mathematical}.
Two important invariants of \eqref{eqn:gpe} are the normalization of the wave function,
\begin{equation}\label{eqn:gpeconstraint}
\int_{\mathbb{R}^{d}} |\psi(\bm x,t)|^2 \mathrm{d}\bm x=1,
\end{equation}
and the energy per particle,
\begin{equation}\label{eqn:gpeenergy}
E(\psi(\cdot,t))=\int_{\mathbb{R}^{d}} \left[\dfrac12 |\nabla \psi(\bm x,t)|^2+V(\bm x)| \psi(\bm x,t)|^2+ \dfrac\beta2 |\psi(\bm x,t)|^4\right]\mathrm{d}\bm x.
\end{equation}
To find a stationary solution to \eqref{eqn:gpe}, substituting $\psi(\bm x,t)=\mathrm{e}^{-\mathrm{i}\mu t}\phi(\bm x)$ into \eqref{eqn:gpe} gives
\begin{equation}\label{eqn:gpephi}
\mu\;\phi(\bm x)=\left(-\dfrac12 \nabla^2+V(\bm x)+ \beta|\phi(\bm x)|^2\right) \phi(\bm x),\quad \bm x\in \mathbb{R}^{d},
\end{equation}
where $\mu$ is the chemical potential and $\phi(\bm x)$ satisfies a unit complex sphere constraint,
\begin{equation}\label{eqn:becsphereconstraint}
\phi\in \mathcal{M} = \left\{\varphi\in L^2(\mathbb{R}^{d}, \mathbb{C}):
E(\varphi)<\infty,
\int_{\mathbb{R}^{d}} |\varphi(\bm x)|^2 \mathrm{d}\bm x=1\right\}.
\end{equation}
The eigenfunction to the nonlinear elliptic eigenvalue problem \eqref{eqn:gpephi} is the stationary point of the energy
\begin{equation}\label{eqn:becenergy}
E(\phi)=\int_{\mathbb{R}^{d}} \left[\dfrac12 |\nabla \phi(\bm x)|^2+V(\bm x)|\phi(\bm x)|^2+ \dfrac\beta2|\phi(\bm x)|^4\right]\mathrm{d}\bm x.
\end{equation}
with the sphere constraint \eqref{eqn:becsphereconstraint}, and the chemical potential $\mu$ can be calculated as,
\begin{equation}\label{eqn:becchemicalpotential}
\mu=\int_{\mathbb{R}^{d}} \left[\dfrac12 |\nabla \phi(\bm x)|^2+V(\bm x)| \phi(\bm x)|^2+ \beta |\phi(\bm x)|^4\right]\mathrm{d}\bm x = E(\phi) + \int_{\mathbb{R}^{d}}\dfrac\beta2 |\phi(\bm x)|^4\mathrm{d}\bm x.
\end{equation}
The ground state of BEC is defined as the eigenfunction of \eqref{eqn:gpephi} subject to the sphere constraint \eqref{eqn:becsphereconstraint} with the lowest energy \eqref{eqn:becenergy}.
Any eigenfunction of \eqref{eqn:gpephi} subject to the sphere constraint \eqref{eqn:becsphereconstraint} with a higher energy is usually called excited states in the physics literature.

We consider a two-dimensional BEC system with a repulsive interaction parameter $\beta=300$, where the condensates are tightly confined in the other dimension.
The potential $V(\bm x):\mathbb{R}^2\to\mathbb{R}$ is simply chosen as a radial-symmetric harmonic oscillator $V(\bm{x})=\frac{1}{2}|\bm{x}|^2$.
The existence, uniqueness up to a phase factor, and smoothness of the ground state have been well studied in previous researches \cite{bao2013mathematical,lieb2001bosons}, and there are many efficient numerical methods for computing the ground state \cite{bao2013mathematical,bao2006efficient,wu2017regularized}.
Furthermore, there have been some attempts to find several excited states of BEC as well.
By choosing proper initial guesses such as odd functions, a few methods for finding ground states can also be applied to compute some excited states with certain symmetry of BEC \cite{bao2006efficient,wu2017regularized}.
Some excited states of BEC have also been obtained with Newton's methods \cite{law2010stable,law2014dynamic} and deflated continuation algorithms \cite{charalampidis2018computing,boulle2020deflation,charalampidis2020bifurcation}, while how to systematically compute the excited states of BEC remains a huge challenge.

The energy \eqref{eqn:becenergy} implicates some invariance, which leads to zero eigenvalues of Hessians at stationary points.
For a stationary solution $\phi\in\mathcal{M}$, both $\mathrm{e}^{\mathrm{i}\vartheta}\phi(\bm x)$ and $\phi\left(R_\vartheta\bm x\right)$ are also stationary solutions with the same index and energy for $\forall \vartheta\in\mathbb{R}$, where $R_\vartheta=\begin{pmatrix}  \cos \vartheta & -\sin \vartheta \\  \sin \vartheta & \cos \vartheta \\\end{pmatrix}$ is the rotation around the origin.
This invariance accounts for two zero eigenvalues for Hessians of general stationary solutions.
As a special case, the ground state and central vortex states can be expressed as $\phi(\bm x)=\mathrm{e}^{\mathrm{i}m\theta}\varphi_m(r)$ using the polar coordinate $(r,\theta)$ and the winding number $m\in\mathbb{Z}$ of the central vortex, so the Hessians at these stationary solutions have only one zero eigenvalue \cite{bao2013mathematical,bao2005ground}.

In the numerical computation, the wave function $\phi$ is truncated into a bounded domain $D=[-M,M]^2$ with homogeneous Dirichlet boundary conditions,
\begin{equation}\label{eqn:becboundary}
\begin{aligned}
E(\phi)&=\int_{D} \left[\dfrac12 |\nabla \phi(\bm x)|^2+V(\bm x)|\phi(\bm x)|^2+ \dfrac\beta2 |\phi(\bm x)|^4\right]\mathrm{d}\bm x,\\
\phi(\bm x) &= 0,\quad \bm x\in \partial D,
\end{aligned}
\end{equation}
because the stationary states decay to zero exponentially fast in the far field from the effect of the trapping potential \cite{bao2013mathematical}.
We discretize the wave function $\phi(\bm x)\in L^2(D)$ with $M=8$ using finite difference methods with $N=128$ nodes along each dimension.
It should be noted that the complex-valued function space $L^2(D)$ is a real Hilbert space with the real inner product,
\begin{equation}\label{eqn:becinnerproduct}
  \langle\phi,\psi\rangle = \int_{D} \dfrac{\phi(\bm x)\bar{\psi}(\bm x)+ \bar{\phi}(\bm x)\psi(\bm x)}{2}\mathrm{d}\bm x,
\end{equation}
and the gradient of \eqref{eqn:gpeenergy}, which is often referred to as the Wirtinger derivative, is
\begin{equation}\label{eqn:becwirtinger}
  \nabla E(\phi) = 2\dfrac{\delta E}{\delta \bar{\phi}}=-\nabla^2 \phi+2V(\bm x)\phi + 2\beta \phi|\phi|^2.
\end{equation}
The retraction operator and the corresponding vector transport for this unit sphere are chosen as \eqref{eqn:exponentialmappingsphere} and \eqref{eqn:paralleltranslationsphere}.

\begin{figure}[htbp]
\centering
\includegraphics[width=\linewidth]{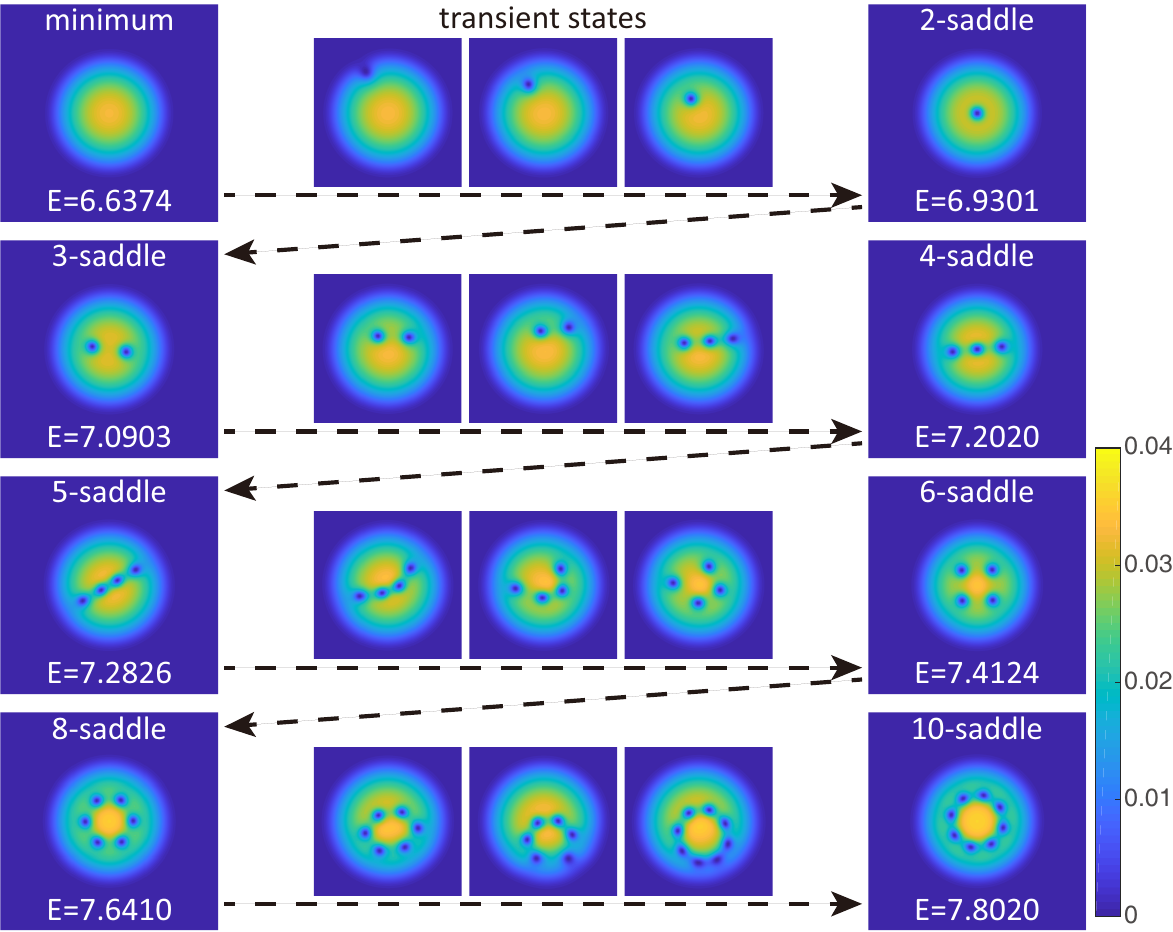}
\caption{An upward pathway sequence from the ground state to a 10-saddle of BEC.
Each dashed arrow represents an upward search to an excited state, and three transient states on some upward dynamical pathways are presented.
We show the probability density $|\phi|^2$ for each state, and present the energy of each stationary solution.}
\label{fig:bec}
\end{figure}

From the ground state, we implement the upward search with fixed step sizes $\alpha^{(n)}=10^{-6}$ and $\beta^{(n)}=10^{-3}$ to search for excited states, which is slightly different from Algorithm~\ref{alg:upward} because of the zero eigenvalues of Hessians at stationary solutions.
For each stationary solution $\phi_i$, we let $\eta_{i,1},\cdots,\eta_{i,m} \in T(\phi_i)$ denote $m$ orthonormal eigenvectors of $\operatorname{Hess}E(\hat{\bm x})$ corresponding to the smallest $m$ eigenvalues $\lambda_{i,1} \leqslant \cdots \leqslant \lambda_{i,m}$.
Since the Hessian at the ground state $\phi_0$ has only one zero eigenvalue $\lambda_{0,1} =0$, the 1-CHiSD from the ground state cannot climb out of the basin.
Alternatively, we treat this degenerate minimum $\phi_0$ as a 1-saddle by regarding the zero eigenvector $\eta_{i,1}$ as an unstable direction, and a 2-CHiSD from $\left(R_{\phi_0}(\varepsilon \eta_{0,2}), \eta_{0,1}, \eta_{0,2}\right)$ finally converges to a 2-saddle $\phi_2$.
This 2-CHiSD upward search involves a quantized vortex of winding number $+1$ moving from the domain edge to the center, and the 2-saddle $\phi_2$ is a central vortex state with only one zero eigenvalue as shown in Fig.~\ref{fig:bec}.
Here, the jump of Morse index comes from the degeneracy of this BEC system, that is, zero eigenvalues of the Riemannian Hessians at the critical points.
Similarly, we implement upward search from $\phi_2$ with a 4-CHiSD from $\left(R_{\phi_2}(\varepsilon \eta_{2,4}), \eta_{2,1}, \cdots, \eta_{2,4}\right)$ and find a degenerate 3-saddle $\phi_3$.
The 3-saddle $\phi_3$ has two vortices with winding numbers $+1$ and $-1$ respectively, and its Hessian has two zero eigenvalues $\lambda_{3,4}=\lambda_{3,5}=0$.
Therefore, the upward search from $\phi_3$ is implemented with a 6-CHiSD from $\left(R_{\phi_3}(\varepsilon \eta_{3,6}), \eta_{3,1}, \cdots, \eta_{3,6}\right)$ to find an excited state with a higher index, which turns out to be a 4-saddle.
This 6-CHiSD upward search involves a new quantized vortex, and the three vortices of the 4-saddle are aligned.
This upward search procedure can be repeated to newly-found excited states, and Fig.~\ref{fig:bec} shows a sequence of seven upward-search attempts from the ground state eventually to a 10-saddle.
It should be noted that the upward search $(k+3)$-CHiSD from a $k$-saddle finally converges to a $(k+1)$-saddle for $k=3,4,5$, and a $(k+2)$-saddle for $k=6,8$.

\section{Conclusions and discussions}
\label{sec:conclusion}
In this article, we proposed a CHiSD method for searching for index-$k$ saddle points subject to general equality constraints.
Applying the Riemannian gradients and Hessians, we derived the dynamical system with a transformed gradient flow as the formulation of CHiSD.
The linear stability of CHiSD at the index-$k$ constrained saddle points is proved.
In the numerical implementation of CHiSD, the retraction operator and vector transport are introduced to discretize the dynamics.
Combined with the CHiSD method, one can construct the solution landscape on a constrained manifold using downward/upward search algorithms.

We presented two numerical examples as applications of the CHiSD method.
For the Thomson problem, we constructed the solution landscape from a planar configuration. Our results can be regarded as an improvement of the saddle point amounts and the kinetic transition networks \cite{mehta2016kinetic,mehta2015exploring}.
Although only cases with $5$, $7$ and $9$ particles are considered here, our method can be applied to study more particles in a straightforward way.
The other application is to search for the excited states of BEC, which correspond to a nonlinear eigenvalue problem.
The upward search algorithm is slightly varied in this case due to the zero eigenvalues of Hessians at stationary points.
We calculated the excited states using the upward search algorithm, which can further applied to the 10-saddle, as it is conjectured that the nonlinear eigenvalue problem \eqref{eqn:gpephi} admits infinitely many eigenfunctions which are linearly independent \cite{bao2005ground}.
Furthermore, in a rotating BEC system, vortices are energetically favored above a critical rotational frequency and this system can admit multiple stable/metastable states \cite{bao2013mathematical,bao2005ground}.
We will systematically explore the solution landscape of rotating BEC and identify probable mechanisms of excitation in a subsequent work.

There are naturally some issues worthy of further investigations in order to improve numerical efficiency.
Since the computational costs mainly depend on calculations of the eigenvectors of Hessians, one may consider more efficient methods for computing the eigenvectors.
For instance, if the problem is ill-conditioned, the power iteration method \eqref{eqn:viteration} may present poor approximations of eigenvectors, leading to a failure of convergence.
Thus, other efficient numerical algorithms will be needed to deal with such difficulties.
Moreover, the fixed step size is used in current version of the CHiSD method for simplicity, which can be adopted by adaptive step sizes to accelerate the convergent rate.
Furthermore, at the end of CHiSD iterations, we can apply the second-order methods to accelerate the convergence.
Another important task is to apply the CHiSD method to explore more practical applications, in which the framework of CHiSD will be adapted to the numerical schemes originally designed for the gradient dynamics.

\begin{acknowledgements}
We would like to thank Professors Qiang Du, Weizhu Bao, and Yongyong Cai for helpful discussions.
J. Yin acknowledges the support from the Elite Program of Computational and Applied Mathematics for Ph.D. Candidates in Peking University.
Z. Huang is supported by the Elite Undergraduate Training Program of School of Mathematical Sciences in Peking University.
\end{acknowledgements}

\section*{Conflict of interest}
The authors declare that they have no conflict of interest.

\bibliographystyle{spmpsci}
\bibliography{chisd}
\end{document}